\documentclass[11pt]{preprint}
\usepackage[full]{textcomp}
\usepackage[osf]{newtxtext} 
\usepackage[cal=boondoxo]{mathalfa}
\usepackage{colortbl}

\usepackage{tikz-cd}
\usetikzlibrary{cd}
\usepackage{comment}

\usepackage{amssymb}
\usepackage{mathtools}
\usepackage{hyperref}
\usepackage{breakurl}
\usepackage{mhenvs}
\usepackage{mhequ} 
\usepackage{mhsymb}
\usepackage{booktabs}
\usepackage{tikz}
\usepackage{tcolorbox}
\usepackage{mathrsfs}
\usepackage[utf8]{inputenc}
\usepackage{longtable}
\usepackage{wrapfig}
\usepackage{rotating}
\usepackage{microtype}
\usepackage{comment}
\usepackage{wasysym}
\usepackage{centernot}
\usepackage{enumitem}
\usepackage{bm}
\usepackage{stackrel}
\usepackage{graphicx}
\usepackage{diagrams}
\makeatletter
\newcommand{\globalcolor}[1]{%
  \color{#1}\global\let\default@color\current@color
}
\makeatother

\usetikzlibrary{calc}
\usetikzlibrary{decorations}
\usetikzlibrary{positioning}
\usetikzlibrary{shapes}

\newif\ifdark
\darkfalse

\ifdark

\definecolor{darkred}{rgb}{0.9,0.2,0.2}
\definecolor{darkblue}{rgb}{0.7,0.3,1}
\definecolor{darkgreen}{rgb}{0.1,0.9,0.1}
\definecolor{franck}{rgb}{0,0.8,1}
\definecolor{pagebackground}{rgb}{.15,.21,.18}
\definecolor{pageforeground}{rgb}{.84,.84,.85}
\pagecolor{pagebackground}
\AtBeginDocument{\globalcolor{pageforeground}}
\definecolor{symbols}{rgb}{0,0.7,1}
\colorlet{connection}{red!80!black}
\colorlet{boxcolor}{blue!50}

\else

\definecolor{darkred}{rgb}{0.7,0.1,0.1}
\definecolor{darkblue}{rgb}{0.4,0.1,0.8}
\definecolor{darkgreen}{rgb}{0.1,0.7,0.1}
\definecolor{franck}{rgb}{0,0,1}
\definecolor{pagebackground}{rgb}{1,1,1}
\definecolor{pageforeground}{rgb}{0,0,0}

\colorlet{symbols}{blue!90!black}
\colorlet{connection}{red!30!black}
\colorlet{boxcolor}{blue!50!black}

\fi

\def\slash{\leavevmode\unskip\kern0.18em/\penalty\exhyphenpenalty\kern0.18em}
\def\dash{\leavevmode\unskip\kern0.18em--\penalty\exhyphenpenalty\kern0.18em}

\DeclareMathAlphabet{\mathbbm}{U}{bbm}{m}{n}

\DeclareFontFamily{U}{BOONDOX-calo}{\skewchar\font=45 }
\DeclareFontShape{U}{BOONDOX-calo}{m}{n}{
  <-> s*[1.05] BOONDOX-r-calo}{}
\DeclareFontShape{U}{BOONDOX-calo}{b}{n}{
  <-> s*[1.05] BOONDOX-b-calo}{}
\DeclareMathAlphabet{\mcb}{U}{BOONDOX-calo}{m}{n}
\SetMathAlphabet{\mcb}{bold}{U}{BOONDOX-calo}{b}{n}

\setlist{noitemsep,topsep=4pt,leftmargin=1.5em}

\DeclareMathAlphabet{\mathbbm}{U}{bbm}{m}{n}

\DeclareMathAlphabet{\mcb}{U}{BOONDOX-calo}{m}{n}
\SetMathAlphabet{\mcb}{bold}{U}{BOONDOX-calo}{b}{n}
\DeclareFontFamily{U}{mathx}{\hyphenchar\font45}
\DeclareFontShape{U}{mathx}{m}{n}{
      <5> <6> <7> <8> <9> <10>
      <10.95> <12> <14.4> <17.28> <20.74> <24.88>
      mathx10
      }{}
\DeclareSymbolFont{mathx}{U}{mathx}{m}{n}
\DeclareMathSymbol{\bigtimes}{1}{mathx}{"91}

\providecommand{\figures}{false}
{ \ifthenelse{\equal{\figures}{false}} {#1}{\[ {\rm Figure \ missing !} \]} }{}
\def\id{\mathrm{id}}

\def\CH{\mathcal{H}}

\def\CG{\mathcal{G}}

\def\CA{\mathcal{A}}

\def\CC{\mathcal{C}}

\def\CB{\mathcal{B}}

\def\CT{\mathcal{T}}

\tikzstyle{tinydots}=[dash pattern=on \pgflinewidth off \pgflinewidth]
\tikzstyle{superdense}=[dash pattern=on 4pt off 1pt]

\newcommand{\BR}{\bf{BRP}}
\newcommand{\AN}{\bf{ARP}}

\def\Deltam{\Delta^{\!-}}

\newcommand*{\shuffle}
{{\,\begin{sideways}\begin{sideways}\begin{sideways}
$\tiny{\exists}$\end{sideways}\end{sideways}\end{sideways}\,}}
\def\${|\!|\!|}

\newenvironment{DIFnomarkup}{}{} 

\theorembodyfont{\rmfamily}

\newfont{\indic}{bbmss12}

\def\Nabla_#1{\nabla_{\!#1}}

\makeatletter
\pgfdeclareshape{crosscircle}
{
  \inheritsavedanchors[from=circle] 
  \inheritanchorborder[from=circle]
  \inheritanchor[from=circle]{north}
  \inheritanchor[from=circle]{north west}
  \inheritanchor[from=circle]{north east}
  \inheritanchor[from=circle]{center}
  \inheritanchor[from=circle]{west}
  \inheritanchor[from=circle]{east}
  \inheritanchor[from=circle]{mid}
  \inheritanchor[from=circle]{mid west}
  \inheritanchor[from=circle]{mid east}
  \inheritanchor[from=circle]{base}
  \inheritanchor[from=circle]{base west}
  \inheritanchor[from=circle]{base east}
  \inheritanchor[from=circle]{south}
  \inheritanchor[from=circle]{south west}
  \inheritanchor[from=circle]{south east}
  \inheritbackgroundpath[from=circle]
  \foregroundpath{
    \centerpoint%
    \pgf@xc=\pgf@x%
    \pgf@yc=\pgf@y%
    \pgfutil@tempdima=\radius%
    \pgfmathsetlength{\pgf@xb}{\pgfkeysvalueof{/pgf/outer xsep}}%
    \pgfmathsetlength{\pgf@yb}{\pgfkeysvalueof{/pgf/outer ysep}}%
    \ifdim\pgf@xb<\pgf@yb%
      \advance\pgfutil@tempdima by-\pgf@yb%
    \else%
      \advance\pgfutil@tempdima by-\pgf@xb%
    \fi%
    \pgfpathmoveto{\pgfpointadd{\pgfqpoint{\pgf@xc}{\pgf@yc}}{\pgfqpoint{-0.707107\pgfutil@tempdima}{-0.707107\pgfutil@tempdima}}}
    \pgfpathlineto{\pgfpointadd{\pgfqpoint{\pgf@xc}{\pgf@yc}}{\pgfqpoint{0.707107\pgfutil@tempdima}{0.707107\pgfutil@tempdima}}}
    \pgfpathmoveto{\pgfpointadd{\pgfqpoint{\pgf@xc}{\pgf@yc}}{\pgfqpoint{-0.707107\pgfutil@tempdima}{0.707107\pgfutil@tempdima}}}
    \pgfpathlineto{\pgfpointadd{\pgfqpoint{\pgf@xc}{\pgf@yc}}{\pgfqpoint{0.707107\pgfutil@tempdima}{-0.707107\pgfutil@tempdima}}}
  }
}
\makeatother

\def\symbol#1{\textcolor{symbols}{#1}}

\def\decorate#1#2{
        \ifnum#2>0
    		\foreach \count in {1,...,#2}{
	       	let
				\p1 = (sourcenode.center),
                \p2 = (sourcenode.east),
				\n1 = {\x2-\x1},
				\n2 = {1mm},
				\n3 = {(1.3+0.6*(\count-1))*\n1},
				\n4 = {0.7*\n1}
			in 
        		node[rectangle,fill=symbols,rotate=30,inner sep=0pt,minimum width=0.2*\n2,minimum height=\n2] at ($(sourcenode.center) + (\n3,\n4)$) {}
				}
		\fi
        \ifnum#1>0
    		\foreach \count in {1,...,#1}{
	       	let
				\p1 = (sourcenode.center),
                \p2 = (sourcenode.east),
				\n1 = {\x2-\x1},
				\n2 = {1mm},
				\n3 = {(1.3+0.6*(\count-1))*\n1},
				\n4 = {0.7*\n1}
			in 
        		node[rectangle,fill=symbols,rotate=-30,inner sep=0pt,minimum width=0.2*\n2,minimum height=\n2] at ($(sourcenode.center) + (-\n3,\n4)$) {}
				}
		\fi
}

\tikzset{
    dectriangle/.style 2 args={
        triangle,
        alias=sourcenode,
        append after command={\decorate{#1}{#2}}
    },
    dectriangle/.default={0}{0},
}

\tikzset{
	cross/.style={path picture={ 
  		\draw[symbols]
			(path picture bounding box.south east) -- (path picture bounding box.north west) (path picture bounding box.south west) -- (path picture bounding box.north east);
		}},
root/.style={circle,fill=green!50!black,inner sep=0pt, minimum size=1.2mm},
        dot/.style={circle,fill=pageforeground,inner sep=0pt, minimum size=1mm},
        blank/.style={circle,fill=white,inner sep=0pt, minimum size=1mm},
        dotred/.style={circle,fill=pageforeground!50!pagebackground,inner sep=0pt, minimum size=2mm},
        var/.style={circle,fill=pageforeground!10!pagebackground,draw=pageforeground,inner sep=0pt, minimum size=3mm},
        kernel/.style={semithick,shorten >=2pt,shorten <=2pt},
        kernels/.style={snake=zigzag,shorten >=2pt,shorten <=2pt,segment amplitude=1pt,segment length=4pt,line before snake=2pt,line after snake=5pt,},
        rho/.style={densely dashed,semithick,shorten >=2pt,shorten <=2pt},
           testfcn/.style={dotted,semithick,shorten >=2pt,shorten <=2pt},
        renorm/.style={shape=circle,fill=pagebackground,inner sep=1pt},
        labl/.style={shape=rectangle,fill=pagebackground,inner sep=1pt},
        xic/.style={very thin,circle,draw=symbols,fill=symbols,inner sep=0pt,minimum size=1.2mm},
        g/.style={very thin,rectangle,draw=symbols,fill=symbols!10!pagebackground,inner sep=0pt,minimum width=2.5mm,minimum height=1.2mm},
        xi/.style={very thin,circle,draw=symbols,fill=symbols!10!pagebackground,inner sep=0pt,minimum size=1.2mm},
	xies/.style={very thin,rectangle,fill=green!50!black!25,draw=symbols,inner sep=0pt,minimum size=1.1mm},
	xiesf/.style={very thin,rectangle,fill=green!50!black,draw=symbols,inner sep=0pt,minimum size=1.1mm},
        xix/.style={very thin,crosscircle,fill=symbols!10!pagebackground,draw=symbols,inner sep=0pt,minimum size=1.2mm},
        X/.style={very thin,cross,rectangle,fill=pagebackground,draw=symbols,inner sep=0pt,minimum size=1.2mm},
	xib/.style={thin,circle,fill=symbols!10!pagebackground,draw=symbols,inner sep=0pt,minimum size=1.6mm},
	xie/.style={thin,circle,fill=green!50!black,draw=symbols,inner sep=0pt,minimum size=1.6mm},
	xid/.style={thin,circle,fill=symbols,draw=symbols,inner sep=0pt,minimum size=1.6mm},
	xibx/.style={thin,crosscircle,fill=symbols!10!pagebackground,draw=symbols,inner sep=0pt,minimum size=1.6mm},
	kernels2/.style={very thick,draw=connection,segment length=12pt},
	keps/.style={thin,draw=symbols,->},
	kepspr/.style={thick,draw=connection,->},
	krho/.style={thin,draw=symbols,superdense,->},
	krhopr/.style={thick,draw=connection,superdense,->},
	triangle/.style = { regular polygon, regular polygon sides=3},
	not/.style={thin,circle,draw=connection,fill=connection,inner sep=0pt,minimum size=0.5mm},
	diff/.style = {very thin,draw=symbols,triangle,fill=red!50!black,inner sep=0pt,minimum size=1.6mm},
	diff1/.style = {very thin,dectriangle={1}{0},fill=red!50!black,draw=symbols,inner sep=0pt,minimum size=1.6mm},
	diff2/.style = {very thin,dectriangle={1}{1},fill=red!50!black,draw=symbols,inner sep=0pt,minimum size=1.6mm},
		diffmini/.style = {very thin,rectangle,fill=black,draw=black,inner sep=0pt,minimum size=0.75mm},
	 kernelsmod/.style={very thick,draw=connection,segment length=12pt},
	 rec/.style = {very thin,rectangle,fill=black,draw=black,inner sep=0pt,minimum size=2mm},
	cerc/.style={very thin,circle,draw=black,fill=symbols,inner sep=0pt,minimum size=2mm},
	stars/.style={very thin,star,star points=6,star point ratio=0.5, draw=black,fill=red,inner sep=0pt,minimum size=0.7mm},
	>=stealth,
        }

\makeatletter
\def\DeclareSymbol#1#2#3{%
	\expandafter\gdef\csname MH@symb@#1\endcsname{\tikzsetnextfilename{symbol#1}%
	\tikz[baseline=#2,scale=0.15,draw=symbols,line join=round]{#3}}%
	\expandafter\gdef\csname MH@symb@#1s\endcsname{\scalebox{0.75}{\tikzsetnextfilename{symbol#1}%
	\tikz[baseline=#2,scale=0.15,draw=symbols,line join=round]{#3}}}%
	\expandafter\gdef\csname MH@symb@#1ss\endcsname{\scalebox{0.65}{\tikzsetnextfilename{symbol#1}%
	\tikz[baseline=#2,scale=0.15,draw=symbols,line join=round]{#3}}}%
	}
\def\<#1>{\ifthenelse{\boolean{mmode}}{\mathchoice{\csname MH@symb@#1\endcsname}{\csname MH@symb@#1\endcsname}{\csname MH@symb@#1s\endcsname}{\csname MH@symb@#1ss\endcsname}}{\csname MH@symb@#1\endcsname}}
\makeatother

\DeclareSymbol{Xi22}{0.5}{\draw (0,0) node[xi] {} -- (-1,1) node[not] {} -- (0,2) node[xi] {};} 
\DeclareSymbol{Xi2}{-2}{\draw (-1,-0.25) node[xi] {} -- (0,1) node[xi] {};} 
\DeclareSymbol{Xi2b}{-2}{\draw (-1,-0.25) node[xic] {} -- (0,1) node[xic] {};} 
\DeclareSymbol{Xi2g}{-2}{\draw (-1,-0.25) node[xies] {} -- (0,1) node[xi] {};} 
\DeclareSymbol{Xi2g2}{-2}{\draw (-1,-0.25) node[xi] {} -- (0,1) node[xies] {};} 
\DeclareSymbol{cXi2}{-2}{\draw (0,-0.25) node[xi] {} -- (-1,1) node[xic] {};}
\DeclareSymbol{Xi3}{0}{\draw (0,0) node[xi] {} -- (-1,1) node[xi] {} -- (0,2) node[xi] {};}
\DeclareSymbol{XiIIXi}{0}{\draw (0,0) node[xi] {} -- (-1,1); \draw[kernels2] (-1,1) node[not] {} -- (0,2) node[xi] {};}

\DeclareSymbol{Xi4}{2}{\draw (0,0) node[xi] {} -- (-1,1) node[xi] {} -- (0,2) node[xi] {} -- (-1,3) node[xi] {};}
\DeclareSymbol{Xi4_1}{2}{\draw (0,0) node[xic] {} -- (-1,1) node[xic] {} -- (0,2) node[xi] {} -- (-1,3) node[xi] {};}
\DeclareSymbol{Xi4_2}{2}{\draw (0,0) node[xic] {} -- (-1,1) node[xi] {} -- (0,2) node[xi] {} -- (-1,3) node[xic] {};}
\DeclareSymbol{Xi2X}{-2}{\draw (0,-0.25) node[xi] {} -- (-1,1) node[xix] {};}
\DeclareSymbol{XXi2}{-2}{\draw (0,-0.25) node[xix] {} -- (-1,1) node[xi] {};}
\DeclareSymbol{IIXi}{0}{\draw (0,-0.25) node[not] {} -- (-1,1) node[xi] {} -- (0,2) node[xi] {};}
\DeclareSymbol{IXi^2}{-1}{\draw (-1,1) node[xi] {} -- (0,0) node[not] {} -- (1,1) node[xi] {};}
\DeclareSymbol{IIXi^2}{-4}{\draw (0,-1.5) node[not] {} -- (0,0);
\draw[kernels2] (-1,1) node[xi] {} -- (0,0) node[not] {} -- (1,1) node[xi] {};}
\DeclareSymbol{XiX}{-2.8}{\node[xibx] {};}
\DeclareSymbol{tauX}{-2.8}{ \node[X] {};}
\DeclareSymbol{Xi}{-2.8}{\node[xib] {};}

\DeclareSymbol{IXiX}{-1}{\draw (0,-0.25) node[not] {} -- (-1,1) node[xix] {};}
\DeclareSymbol{IXi3}{2}{\draw (0,-0.25) node[not] {} -- (-1,1) node[xi] {} -- (0,2) node[xi] {} -- (-1,3) node[xi] {};}
\DeclareSymbol{IXi}{-2}{\draw (0,-0.25) node[not] {} -- (-1,1) node[xi] {};}
\DeclareSymbol{XiI}{-2}{\draw (0,-0.25) node[xi] {} -- (-1,1) node[not] {};}

\DeclareSymbol{Xi4b}{0}{\draw(0,1.5) node[xi] {} -- (0,0); \draw (-1,1) node[xi] {} -- (0,0) node[xi] {} -- (1,1) node[xi] {};}
\DeclareSymbol{Xi4b'}{0}{\draw(0,1.5) node[xi] {} -- (0,-0.2); \draw (-1,1) node[xi] {} -- (0,-0.2) node[not] {} -- (1,1) node[xi] {};}
\DeclareSymbol{Xi4c}{0}{\draw (0,1) -- (0.8,2.2) node[xi] {};\draw (0,-0.25) node[xi] {} -- (0,1) node[xi] {} -- (-0.8,2.2) node[xi] {};}
\DeclareSymbol{Xi4d}{-4.5}{\draw (0,-1.5) node[not] {} -- (0,0); \draw (-1,1) node[xi] {} -- (0,0) node[xi] {} -- (1,1) node[xi] {};}
\DeclareSymbol{Xi4e}{0}{\draw (0,2) node[xi] {} -- (-1,1) node[xi] {} -- (0,0) node[xi] {} -- (1,1) node[xi] {};}
\DeclareSymbol{Xi4e'}{0}{\draw (0,2) node[xi] {} -- (-1,1) node[xi] {} -- (0,-0.2) node[not] {} -- (1,1) node[xi] {};}

\DeclareSymbol{Xitwo}
{0}{\draw[kernels2] (0,0) node[not] {} -- (-1,1) node[not] {}
-- (-2,2) node[not]{} -- (-3,3) node[xi]  {};
\draw[kernels2] (0,0) -- (1,1) node[xi] {};
\draw[kernels2] (-1,1) -- (0,2) node[xi] {};
\draw[kernels2] (-2,2) -- (-1,3) node[xi] {};}

\DeclareSymbol{IXitwo}
{0}{\draw (-.7,1.2) node[xi] {} -- (0,-0.2) -- (.7,1.2) node[xi] {};}
\DeclareSymbol{I1Xitwo}
{0}{\draw[kernels2] (0,0) node[not] {} -- (-1,1) node[xi] {};
\draw[kernels2] (0,0) -- (1,1) node[xi] {};}

\DeclareSymbol{I1Xitwobis}
{0}{\draw[kernels2] (0,0) node[not] {} -- (-1,1) node[xies] {};
\draw[kernels2] (0,0) -- (1,1) node[xies] {};}

\DeclareSymbol{I1Xitwog}
{0}{\draw[kernels2] (0,0) node[not] {} -- (-1,1) node[xies] {};
\draw[kernels2] (0,0) -- (1,1) node[xi] {};}

\DeclareSymbol{cI1Xitwo}
{0}{\draw[kernels2] (0,0) node[not] {} -- (-1,1) node[xic] {};
\draw[kernels2] (0,0) -- (1,1) node[xi] {};}

\DeclareSymbol{I1IXi3}{0}{\draw (0,0) node[xi] {} -- (-1,1) ; 
\draw[kernels2] (-1,1) node[not] {} -- (0,2) node[xi] {};
\draw[kernels2] (-1,1) node[not] {} -- (-2,2) node[xi] {};}

\DeclareSymbol{I1Xi3c}{-1}{\draw[kernels2](0,1.5) node[xi] {} -- (0,0) node[not] {}; \draw (-1,1) node[xi] {} -- (0,0) ; \draw[kernels2] (0,0) -- (1,1) node[xi] {};}

\DeclareSymbol{I1Xi3cbis}{-1}{\draw[kernels2](0,1.5) node[xies] {} -- (0,0) node[not] {}; \draw (-1,1) node[xies] {} -- (0,0) ; \draw[kernels2] (0,0) -- (1,1) node[xies] {};}

\DeclareSymbol{I1IXi3b}{0}{\draw[kernels2] (0,0) node[not] {} -- (-1,1) ; \draw[kernels2] (0,0)   -- (1,1) node[xi] {} ;
\draw (-1,1) node[xi] {} -- (0,2) node[xi] {};
}

\DeclareSymbol{I1IXi3c}{0}{\draw[kernels2] (0,0) node[not] {} -- (-1,1) ; \draw[kernels2] (0,0)   -- (1,1) node[xi] {} ;
\draw[kernels2] (-1,1) node[not] {} -- (0,2) node[xi] {};
\draw[kernels2] (-1,1) node[not] {} -- (-2,2) node[xi] {};}

\DeclareSymbol{I1IXi3cbis}{0}{\draw[kernels2] (0,0) node[not] {} -- (-1,1) ; \draw[kernels2] (0,0)   -- (1,1) node[xies] {} ;
\draw[kernels2] (-1,1) node[not] {} -- (0,2) node[xies] {};
\draw[kernels2] (-1,1) node[not] {} -- (-2,2) node[xies] {};}

\DeclareSymbol{I1Xi}{0}{\draw[kernels2] (0,0) node[not] {} -- (-1,1)  node[xi] {} ;}

\DeclareSymbol{I1Xi4a}{2}{\draw[kernels2] (0,0) node[not] {} -- (-1,1) ; \draw[kernels2] (0,0) node[not] {} -- (1,1) node[xi] {} ;
\draw (-1,1) node[xi] {} -- (0,2) node[xi] {} -- (-1,3) node[xi] {};}

\DeclareSymbol{cI1Xi4a}{2}{\draw[kernels2] (0,0) node[not] {} -- (-1,1) ; \draw[kernels2] (0,0) node[not] {} -- (1,1) node[xic] {} ;
\draw (-1,1) node[xic] {} -- (0,2) node[xi] {} -- (-1,3) node[xi] {};}

\DeclareSymbol{I1Xi4b}{2}{\draw (0,0) node[xi] {} -- (-1,1) node[xi] {} -- (0,2) ; \draw[kernels2] (0,2) node[not] {} -- (-1,3) node[xi] {};\draw[kernels2] (0,2)  -- (1,3) node[xi] {};
}

\DeclareSymbol{cI1Xi4b}{2}{\draw (0,0) node[xic] {} -- (-1,1) node[xic] {} -- (0,2) ; \draw[kernels2] (0,2) node[not] {} -- (-1,3) node[xi] {};\draw[kernels2] (0,2)  -- (1,3) node[xi] {};
}

\DeclareSymbol{I1Xi4c}{2}{\draw (0,0) node[xi] {} -- (-1,1) node[not] {}; \draw[kernels2] (-1,1) -- (0,2) ; 
\draw[kernels2] (-1,1) -- (-2,2) node[xi] {} ;
\draw (0,2) node[xi] {} -- (-1,3) node[xi] {};}

\DeclareSymbol{cI1Xi4c}{2}{\draw (0,0) node[xic] {} -- (-1,1) node[not] {}; \draw[kernels2] (-1,1) -- (0,2) ; 
\draw[kernels2] (-1,1) -- (-2,2) node[xic] {} ;
\draw (0,2) node[xi] {} -- (-1,3) node[xi] {};}

\DeclareSymbol{I1Xi4ab}{2}{\draw[kernels2] (0,0) node[not] {} -- (-1,1) ; \draw[kernels2] (0,0) node[not] {} -- (1,1) node[xi] {};\draw (-1,1) node[xi] {} -- (0,2) ; \draw[kernels2] (0,2) node[not] {} -- (-1,3) node[xi] {};\draw[kernels2] (0,2)  -- (1,3) node[xi] {}; }

\DeclareSymbol{cI1Xi4ab}{2}{\draw[kernels2] (0,0) node[not] {} -- (-1,1) ; \draw[kernels2] (0,0) node[not] {} -- (1,1) node[xic] {};\draw (-1,1) node[xic] {} -- (0,2) ; \draw[kernels2] (0,2) node[not] {} -- (-1,3) node[xi] {};\draw[kernels2] (0,2)  -- (1,3) node[xi] {}; }

\DeclareSymbol{I1Xi4bc}{2}{\draw (0,0) node[xi] {} -- (-1,1) node[not] {}; \draw[kernels2] (-1,1) -- (0,2) ; 
\draw[kernels2] (-1,1) -- (-2,2) node[xi] {} ; \draw[kernels2] (0,2) node[not] {} -- (-1,3) node[xi] {};\draw[kernels2] (0,2)  -- (1,3) node[xi] {};
}

\DeclareSymbol{cI1Xi4bc}{2}{\draw (0,0) node[xic] {} -- (-1,1) node[not] {}; \draw[kernels2] (-1,1) -- (0,2) ; 
\draw[kernels2] (-1,1) -- (-2,2) node[xic] {} ; \draw[kernels2] (0,2) node[not] {} -- (-1,3) node[xi] {};\draw[kernels2] (0,2)  -- (1,3) node[xi] {};
}

\DeclareSymbol{I1Xi4abcc1}{2}{\draw[kernels2] (0,0) node[not] {} -- (-1,1) node[not] {}
-- (-2,2) node[not]{} -- (-3,3) node[xic]  {};
\draw[kernels2] (0,0) -- (1,1) node[xic] {};
\draw[kernels2] (-1,1) -- (0,2) node[xi] {};
\draw[kernels2] (-2,2) -- (-1,3) node[xi] {};
}

\DeclareSymbol{I1Xi4abcc1b}{2}{\draw[kernels2] (0,0) node[not] {} -- (-1,1) node[not] {}
-- (-2,2) node[not]{} -- (-3,3) node[xi]  {};
\draw[kernels2] (0,0) -- (1,1) node[xic] {};
\draw[kernels2] (-1,1) -- (0,2) node[xic] {};
\draw[kernels2] (-2,2) -- (-1,3) node[xi] {};
}

\DeclareSymbol{I1Xi4abcc2}{2}{\draw[kernels2] (0,0) node[not] {} -- (-1,1) node[not] {}
-- (-2,2) node[not]{} -- (-3,3) node[xic]  {};
\draw[kernels2] (0,0) -- (1,1) node[xi] {};
\draw[kernels2] (-1,1) -- (0,2) node[xi] {};
\draw[kernels2] (-2,2) -- (-1,3) node[xic] {};
}

\DeclareSymbol{I1Xi4ac}{2}{\draw[kernels2] (0,0) node[not] {} -- (-1,1) ; \draw[kernels2] (0,0) node[not] {} -- (1,1) node[xi] {}; 
\draw[kernels2] (-1,1) node[not] {} -- (0,2) ; 
\draw[kernels2] (-1,1) -- (-2,2) node[xi] {} ;
\draw (0,2) node[xi] {} -- (-1,3) node[xi] {};}

\DeclareSymbol{cI1Xi4ac}{2}{\draw[kernels2] (0,0) node[not] {} -- (-1,1) ; \draw[kernels2] (0,0) node[not] {} -- (1,1) node[xic] {}; 
\draw[kernels2] (-1,1) node[not] {} -- (0,2) ; 
\draw[kernels2] (-1,1) -- (-2,2) node[xic] {} ;
\draw (0,2) node[xi] {} -- (-1,3) node[xi] {};}

\DeclareSymbol{I1Xi4acc1}{2}{\draw[kernels2] (0,0) node[not] {} -- (-1,1) ; \draw[kernels2] (0,0) node[not] {} -- (1,1) node[xic] {}; 
\draw[kernels2] (-1,1) node[not] {} -- (0,2) ; 
\draw[kernels2] (-1,1) -- (-2,2) node[xi] {} ;
\draw (0,2) node[xic] {} -- (-1,3) node[xi] {};}

\DeclareSymbol{I1Xi4acc2}{2}{\draw[kernels2] (0,0) node[not] {} -- (-1,1) ; \draw[kernels2] (0,0) node[not] {} -- (1,1) node[xic] {}; 
\draw[kernels2] (-1,1) node[not] {} -- (0,2) ; 
\draw[kernels2] (-1,1) -- (-2,2) node[xi] {} ;
\draw (0,2) node[xi] {} -- (-1,3) node[xic] {};}

\DeclareSymbol{2I1Xi4}{2}{\draw[kernels2] (0,0) node[not] {} -- (-1,1) node[not] {};
\draw[kernels2] (0,0) -- (1,1) node[not] {};
\draw[kernels2] (-1,1) -- (-1.5,2.5) node[xi] {};
\draw[kernels2] (-1,1) -- (-0.5,2.5) node[xi] {};
\draw[kernels2] (1,1) -- (0.5,2.5) node[xi] {};
\draw[kernels2] (1,1) -- (1.5,2.5) node[xi] {};
}

\DeclareSymbol{2I1Xi4dis}{2}{\draw[kernels2] (0,0) node[not] {} -- (-1,1) node[not] {};
\draw[kernels2] (0,0) -- (1,1) node[not] {};
\draw[kernels2] (-1,1) -- (-1.5,2.5) node[xies] {};
\draw[kernels2] (-1,1) -- (-0.5,2.5) node[xies] {};
\draw[kernels2] (1,1) -- (0.5,2.5) node[xies] {};
\draw[kernels2] (1,1) -- (1.5,2.5) node[xies] {};
}

\DeclareSymbol{2I1Xi4c1}{2}{\draw[kernels2] (0,0) node[not] {} -- (-1,1) node[not] {};
\draw[kernels2] (0,0) -- (1,1) node[not] {};
\draw[kernels2] (-1,1) -- (-1.5,2.5) node[xic] {};
\draw[kernels2] (-1,1) -- (-0.5,2.5) node[xi] {};
\draw[kernels2] (1,1) -- (0.5,2.5) node[xic] {};
\draw[kernels2] (1,1) -- (1.5,2.5) node[xi] {};
}

\DeclareSymbol{2I1Xi4c2}{2}{\draw[kernels2] (0,0) node[not] {} -- (-1,1) node[not] {};
\draw[kernels2] (0,0) -- (1,1) node[not] {};
\draw[kernels2] (-1,1) -- (-1.5,2.5) node[xic] {};
\draw[kernels2] (-1,1) -- (-0.5,2.5) node[xic] {};
\draw[kernels2] (1,1) -- (0.5,2.5) node[xi] {};
\draw[kernels2] (1,1) -- (1.5,2.5) node[xi] {};
}

\DeclareSymbol{2I1Xi4b}{2}{\draw[kernels2] (0,0) node[not] {} -- (-1,1) ;
\draw[kernels2] (0,0) -- (1,1);
\draw (-1,1) node[xi] {} -- (-1,2.5) node[xi] {};
\draw (1,1)  node[xi] {} -- (1,2.5) node[xi] {};
}

\DeclareSymbol{2I1Xi4bb}{2}{\draw[kernels2] (0,0) node[not] {} -- (-1,1) ;
\draw[kernels2] (0,0) -- (1,1);
\draw (-1,1) node[xi] {} -- (-1,2.5) node[xiesf] {};
\draw (1,1)  node[xi] {} -- (1,2.5) node[xic] {};
}

\DeclareSymbol{2I1Xi4c}{2}{\draw[kernels2] (0,0) node[not] {} -- (-1,1);
\draw[kernels2] (0,0) -- (1,1) node[not] {};
\draw (-1,1)  node[xi] {} -- (-1,2.5) node[xi] {};
\draw[kernels2] (1,1) -- (0.4,2.5) node[xi] {};
\draw[kernels2] (1,1) -- (1.6,2.5) node[xi] {};
}

\DeclareSymbol{2I1Xi4cc1}{2}{\draw[kernels2] (0,0) node[not] {} -- (-1,1);
\draw[kernels2] (0,0) -- (1,1) node[not] {};
\draw (-1,1)  node[xic] {} -- (-1,2.5) node[xi] {};
\draw[kernels2] (1,1) -- (0.4,2.5) node[xic] {};
\draw[kernels2] (1,1) -- (1.6,2.5) node[xi] {};
}

\DeclareSymbol{2I1Xi4cc2}{2}{\draw[kernels2] (0,0) node[not] {} -- (-1,1);
\draw[kernels2] (0,0) -- (1,1) node[not] {};
\draw (-1,1)  node[xic] {} -- (-1,2.5) node[xic] {};
\draw[kernels2] (1,1) -- (0.4,2.5) node[xi] {};
\draw[kernels2] (1,1) -- (1.6,2.5) node[xi] {};
}

\DeclareSymbol{Xi4ba}{0}{\draw(-0.5,1.5) node[xi] {} -- (0,0); \draw (-1.5,1) node[xi] {} -- (0,0) node[not] {}; \draw[kernels2] (0,0) -- (1.5,1) node[xi] {};
\draw[kernels2] (0,0) -- (0.5,1.5) node[xi] {} ;}

\DeclareSymbol{Xi4badis}{0}{\draw(-0.5,1.5) node[xies] {} -- (0,0); \draw (-1.5,1) node[xies] {} -- (0,0) node[not] {}; \draw[kernels2] (0,0) -- (1.5,1) node[xies] {};
\draw[kernels2] (0,0) -- (0.5,1.5) node[xies] {} ;}

\DeclareSymbol{Xi4ba1}{0}{\draw(-0.5,1.5) node[xi] {} -- (0,0); \draw (-1.5,1) node[xi] {} -- (0,0) node[not] {}; \draw[kernels2] (0,0) -- (1.5,1) node[xic] {};
\draw[kernels2] (0,0) -- (0.5,1.5) node[xic] {} ;}

\DeclareSymbol{Xi4ba1b}{0}{\draw(-0.5,1.5) node[xic] {} -- (0,0); \draw (-1.5,1) node[xic] {} -- (0,0) node[not] {}; \draw[kernels2] (0,0) -- (1.5,1) node[xi] {};
\draw[kernels2] (0,0) -- (0.5,1.5) node[xi] {} ;}

\DeclareSymbol{Xi4ba1bdiff}{0}{\draw(-0.5,1.5) node[xic] {} -- (0,0); \draw (-1.5,1) node[xic] {} -- (0,0) node[not] {}; \draw (0,0) -- (1.5,1) node[xi] {};
\draw (0,0) -- (0.5,1.5) node[xi] {};
\draw(0,0) node[diff] {};}

\DeclareSymbol{Xi4ba1bb}{0}{\draw(-0.5,1.5) node[xic] {} -- (0,0); \draw (-1.5,1) node[xiesf] {} -- (0,0) node[not] {}; \draw[kernels2] (0,0) -- (1.5,1) node[xi] {};
\draw[kernels2] (0,0) -- (0.5,1.5) node[xi] {} ;}

\DeclareSymbol{Xi4ba2}{0}{\draw(-0.5,1.5) node[xi] {} -- (0,0); \draw (-1.5,1) node[xic] {} -- (0,0) node[not] {}; \draw[kernels2] (0,0) -- (1.5,1) node[xi] {};
\draw[kernels2] (0,0) -- (0.5,1.5) node[xic] {} ;}

\DeclareSymbol{Xi4ba2b}{0}{\draw(-0.5,1.5) node[xi] {} -- (0,0); \draw (-1.5,1) node[xic] {} -- (0,0) node[not] {}; \draw[kernels2] (0,0) -- (1.5,1) node[xi] {};
\draw[kernels2] (0,0) -- (0.5,1.5) node[xiesf] {} ;}

\DeclareSymbol{Xi4ca}{0}{\draw (0,1) -- (-1,2.2) node[xi] {};\draw (0,-0.25) node[xi] {} -- (0,1) ; \draw[kernels2] (0,1) node[not] {} -- (1,2.2) node[xi] {};
\draw[kernels2] (0,1) {} -- (0,2.7) node[xi] {};
}

\DeclareSymbol{Xi4cb}{0}{\draw (-1,1) -- (-2,2) node[xi] {};\draw[kernels2] (0,0)  -- (-1,1) node[xi] {} ; \draw[kernels2] (0,0) node[not] {} -- (1,1) node[xi] {} ; 
\draw (-1,1) node[xi] {} -- (0,2) node[xi] {};}

\DeclareSymbol{Xi4cbb}{0}{\draw (-1,1) -- (-2,2) node[xiesf] {};\draw[kernels2] (0,0)  -- (-1,1) node[xi] {} ; \draw[kernels2] (0,0) node[not] {} -- (1,1) node[xi] {} ; 
\draw (-1,1) node[xi] {} -- (0,2) node[xic] {};}

\DeclareSymbol{Xi4cbc1}{0}{\draw (-1,1) -- (-2,2) node[xic] {};\draw[kernels2] (0,0)  -- (-1,1) node[xic] {} ; \draw[kernels2] (0,0) node[not] {} -- (1,1) node[xi] {} ; 
\draw (-1,1) node[xic] {} -- (0,2) node[xi] {};}

\DeclareSymbol{Xi4cbc2}{0}{\draw (-1,1) -- (-2,2) node[xi] {};\draw[kernels2] (0,0)  -- (-1,1) node[xi] {} ; \draw[kernels2] (0,0) node[not] {} -- (1,1) node[xic] {} ; 
\draw (-1,1) node[xic] {} -- (0,2) node[xi] {};}

\DeclareSymbol{Xi4cab}{0}{\draw (-1,1) -- (-2,2) node[xi] {};\draw[kernels2] (0,0)  -- (-1,1); \draw[kernels2] (0,0) node[not] {} -- (1,1) node[xi] {} ; 
\draw[kernels2] (-1,1)  {} -- (0,2) node[xi] {};
\draw[kernels2] (-1,1) node[not] {} -- (-1,2.5) node[xi] {};
}

\DeclareSymbol{Xi4cabdis}{0}{\draw (-1,1) -- (-2,2) node[xies] {};\draw[kernels2] (0,0)  -- (-1,1); \draw[kernels2] (0,0) node[not] {} -- (1,1) node[xies] {} ; 
\draw[kernels2] (-1,1)  {} -- (0,2) node[xies] {};
\draw[kernels2] (-1,1) node[not] {} -- (-1,2.5) node[xies] {};
}

\DeclareSymbol{Xi4cabc1}{0}{\draw (-1,1) -- (-2,2) node[xi] {};\draw[kernels2] (0,0)  -- (-1,1); \draw[kernels2] (0,0) node[not] {} -- (1,1) node[xic] {} ; 
\draw[kernels2] (-1,1)  {} -- (0,2) node[xic] {};
\draw[kernels2] (-1,1) node[not] {} -- (-1,2.5) node[xi] {};
}

\DeclareSymbol{Xi4cabc2}{0}{\draw (-1,1) -- (-2,2) node[xic] {};\draw[kernels2] (0,0)  -- (-1,1); \draw[kernels2] (0,0) node[not] {} -- (1,1) node[xic] {} ; 
\draw[kernels2] (-1,1)  {} -- (0,2) node[xi] {};
\draw[kernels2] (-1,1) node[not] {} -- (-1,2.5) node[xi] {};
}

\DeclareSymbol{Xi4ea}{1.5}{\draw (-1,2.5) node[xi] {} -- (-1,1) node[xi] {} -- (0,0); 
 \draw[kernels2] (0,0)  -- (1,1) node[xi] {};
\draw[kernels2] (0,0) node[not] {} -- (0,1.5) node[xi] {}; }

\DeclareSymbol{Xi4eac1}{1.5}{\draw (-1,2.5) node[xic] {} -- (-1,1) node[xi] {} -- (0,0); 
 \draw[kernels2] (0,0)  -- (1,1) node[xic] {};
\draw[kernels2] (0,0) node[not] {} -- (0,1.5) node[xi] {}; }

\DeclareSymbol{Xi4eac1b}{1.5}{\draw (-1,2.5) node[xic] {} -- (-1,1) node[xi] {} -- (0,0); 
 \draw[kernels2] (0,0)  -- (1,1) node[xiesf] {};
\draw[kernels2] (0,0) node[not] {} -- (0,1.5) node[xi] {}; }

\DeclareSymbol{Xi4eac2}{1.5}{\draw (-1,2.5) node[xic] {} -- (-1,1) node[xic] {} -- (0,0); 
 \draw[kernels2] (0,0)  -- (1,1) node[xi] {};
\draw[kernels2] (0,0) node[not] {} -- (0,1.5) node[xi] {}; }

\DeclareSymbol{Xi4eact1}{1.5}{\draw (-1,2.5) node[xic] {} -- (-1,1) node[xi] {} -- (0,0); 
 \draw (0,0)  -- (1,1) node[xic] {};
\draw[rho] (0,0) node[not] {} -- (0,1.5) node[xi] {}; }

\DeclareSymbol{Xi4eact2}{1.5}{\draw[rho] (-1,2.5) node[xic] {} -- (-1,1) node[xi] {} -- (0,0); 
 \draw (0,0)  -- (1,1) node[xic] {};
\draw (0,0) node[not] {} -- (0,1.5) node[xi] {}; }

\DeclareSymbol{Xi4eabis}{1.5}{\draw (-1,2.5) node[xi] {} -- (-1,1) ; \draw[kernels2] (-1,1) node[xi] {} -- (0,0); 
 \draw (0,0)  -- (1,1) node[xi] {};
\draw[kernels2] (0,0) node[not] {} -- (0,1.5) node[xi] {}; }

\DeclareSymbol{Xi4eabisc1}{1.5}{\draw (-1,2.5) node[xic] {} -- (-1,1) ; \draw[kernels2] (-1,1) node[xi] {} -- (0,0); 
 \draw (0,0)  -- (1,1) node[xi] {};
\draw[kernels2] (0,0) node[not] {} -- (0,1.5) node[xic] {}; }

\DeclareSymbol{Xi4eabisc1b}{1.5}{\draw (-1,2.5) node[xic] {} -- (-1,1) ; \draw[kernels2] (-1,1) node[xi] {} -- (0,0); 
 \draw (0,0)  -- (1,1) node[xi] {};
\draw[kernels2] (0,0) node[not] {} -- (0,1.5) node[xiesf] {}; }

\DeclareSymbol{Xi4eabisc1bis}{1.5}{\draw (-1,2.5) node[xi] {} -- (-1,1) ; \draw[kernels2] (-1,1) node[xi] {} -- (0,0); 
 \draw (0,0)  -- (1,1) node[xi] {};
\draw[kernels2] (0,0) node[not] {} -- (0,1.5) node[xi] {};
\draw (-2,1) node[] {\tiny{$i$}};
\draw (-2,2.5) node[] {\tiny{$\ell$}};
\draw (2,1) node[] {\tiny{$k$}};
\draw (0,2.5) node[] {\tiny{$j$}};
 }

\DeclareSymbol{Xi4eabisc1tris}{1.5}{\draw (-1,2.5) node[xi] {} -- (-1,1) ; \draw[kernels2] (-1,1) node[xi] {} -- (0,0); 
 \draw (0,0)  -- (1,1) node[xi] {};
\draw[kernels2] (0,0) node[not] {} -- (0,1.5) node[xi] {};
\draw (-2,1) node[] {\tiny{i}};
\draw (-2,2.5) node[] {\tiny{j}};
\draw (2,1) node[] {\tiny{j}};
\draw (0,2.5) node[] {\tiny{i}};
 }

\DeclareSymbol{Xi4eabisc1quater}{1.5}{\draw (-1,2.5) node[xic] {} -- (-1,1) ; \draw[kernels2] (-1,1) node[xi] {} -- (0,0); 
 \draw (0,0)  -- (1,1) node[xic] {};
\draw[kernels2] (0,0) node[not] {} -- (0,1.5) node[xi] {};
 }

\DeclareSymbol{Xi4eabisc2}{1.5}{\draw (-1,2.5) node[xic] {} -- (-1,1) ; \draw[kernels2] (-1,1) node[xi] {} -- (0,0); 
 \draw (0,0)  -- (1,1) node[xic] {};
\draw[kernels2] (0,0) node[not] {} -- (0,1.5) node[xi] {}; }

\DeclareSymbol{Xi4eabisc2l}{1.5}{\draw (-1,2.5) node[xiesf] {} -- (-1,1) ; \draw[kernels2] (-1,1) node[xi] {} -- (0,0); 
 \draw (0,0)  -- (1,1) node[xic] {};
\draw[kernels2] (0,0) node[not] {} -- (0,1.5) node[xi] {}; }

\DeclareSymbol{Xi4eabisc2r}{1.5}{\draw (-1,2.5) node[xic] {} -- (-1,1) ; \draw[kernels2] (-1,1) node[xi] {} -- (0,0); 
 \draw (0,0)  -- (1,1) node[xiesf] {};
\draw[kernels2] (0,0) node[not] {} -- (0,1.5) node[xi] {}; }

\DeclareSymbol{Xi4eabisc3}{1.5}{\draw (-1,2.5) node[xic] {} -- (-1,1) ; \draw[kernels2] (-1,1) node[xic] {} -- (0,0); 
 \draw (0,0)  -- (1,1) node[xi] {};
\draw[kernels2] (0,0) node[not] {} -- (0,1.5) node[xi] {}; }

\DeclareSymbol{Xi4eb}{0}{
\draw[kernels2] (0,2) node[xi] {} -- (-1,1) ; \draw[kernels2] (-2,2)  node[xi] {} -- (-1,1) ; \draw (-1,1)  node[not] {} -- (0,0); 
 \draw (0,0) node[xi] {}  -- (1,1) node[xi] {};
}

\DeclareSymbol{Xi4eab}{1.5}{\draw[kernels2] (-1,2.5) node[xi] {} -- (-1,1) ; \draw[kernels2] (-2,2)  node[xi] {} -- (-1,1) ; \draw (-1,1)  node[not] {} -- (0,0); 
 \draw[kernels2] (0,0)  -- (1,1) node[xi] {};
\draw[kernels2] (0,0) node[not] {} -- (0,1.5) node[xi] {}; 
}

\DeclareSymbol{Xi4eabdis}{1.5}{\draw[kernels2] (-1,2.5) node[xies] {} -- (-1,1) ; \draw[kernels2] (-2,2)  node[xies] {} -- (-1,1) ; \draw (-1,1)  node[not] {} -- (0,0); 
 \draw[kernels2] (0,0)  -- (1,1) node[xies] {};
\draw[kernels2] (0,0) node[not] {} -- (0,1.5) node[xies] {}; 
}

\DeclareSymbol{Xi4eabc1}{1.5}{\draw[kernels2] (-1,2.5) node[xic] {} -- (-1,1) ; \draw[kernels2] (-2,2)  node[xi] {} -- (-1,1) ; \draw (-1,1)  node[not] {} -- (0,0); 
 \draw[kernels2] (0,0)  -- (1,1) node[xic] {};
\draw[kernels2] (0,0) node[not] {} -- (0,1.5) node[xi] {}; 
}

\DeclareSymbol{Xi4eabc2}{1.5}{\draw[kernels2] (-1,2.5) node[xi] {} -- (-1,1) ; \draw[kernels2] (-2,2)  node[xi] {} -- (-1,1) ; \draw (-1,1)  node[not] {} -- (0,0); 
 \draw[kernels2] (0,0)  -- (1,1) node[xic] {};
\draw[kernels2] (0,0) node[not] {} -- (0,1.5) node[xic] {}; 
}

\DeclareSymbol{Xi4eabbis}{1.5}{\draw[kernels2] (-1,2.5) node[xi] {} -- (-1,1) ; \draw[kernels2] (-2,2)  node[xi] {} -- (-1,1) ; \draw[kernels2] (-1,1)  node[not] {} -- (0,0); 
 \draw (0,0)  -- (1,1) node[xi] {};
\draw[kernels2] (0,0) node[not] {} -- (0,1.5) node[xi] {}; 
}

\DeclareSymbol{Xi4eabbisc1}{1.5}{\draw[kernels2] (-1,2.5) node[xic] {} -- (-1,1) ; \draw[kernels2] (-2,2)  node[xi] {} -- (-1,1) ; \draw[kernels2] (-1,1)  node[not] {} -- (0,0); 
 \draw (0,0)  -- (1,1) node[xic] {};
\draw[kernels2] (0,0) node[not] {} -- (0,1.5) node[xi] {}; 
}

\DeclareSymbol{Xi4eabbisc1perm}{1.5}{\draw[kernels2] (-1,2.5) node[xi] {} -- (-1,1) ; \draw[kernels2] (-2,2)  node[xic] {} -- (-1,1) ; \draw[kernels2] (-1,1)  node[not] {} -- (0,0); 
 \draw (0,0)  -- (1,1) node[xic] {};
\draw[kernels2] (0,0) node[not] {} -- (0,1.5) node[xi] {}; 
}

\DeclareSymbol{Xi4eabbisc2}{1.5}{\draw[kernels2] (-1,2.5) node[xi] {} -- (-1,1) ; \draw[kernels2] (-2,2)  node[xi] {} -- (-1,1) ; \draw[kernels2] (-1,1)  node[not] {} -- (0,0); 
 \draw (0,0)  -- (1,1) node[xic] {};
\draw[kernels2] (0,0) node[not] {} -- (0,1.5) node[xic] {}; 
}

\DeclareSymbol{Xi2cbis}{0}{\draw[kernels2] (0,1) -- (0.8,2.2) node[xi] {};\draw[kernels2] (0,-0.25) node[not] {} -- (0,1); \draw[kernels2] (0,1) node[not] {} -- (-0.8,2.2) node[xi] {};}

\DeclareSymbol{Xi2cbis1}{0}{\draw (0,1) -- (-0.8,2.2) node[xi] {};\draw[kernels2] (0,-0.25) node[not] {} -- (0,1) node[xi] {}; }

\DeclareSymbol{Xi2Xbis}{-2}{\draw[kernels2] (0,-0.25)  -- (-1,1) ; \draw (-1,1) node[xix] {};
\draw[kernels2] (0,-0.25) node[not] {} -- (1,1) node[xi] {};}

\DeclareSymbol{XXi2bis}{-2}{\draw[kernels2] (0,-0.25) -- (-1,1) node[xi] {};
\draw[kernels2] (0,-0.25) node[X] {} -- (1,1) node[xi] {};}

\DeclareSymbol{I1XiIXi}{0}{\draw[kernels2] (0,-0.25) -- (1,1) node[xi] {};
\draw (0,-0.25) node[not] {} -- (-1,1) node[xi] {};}

\DeclareSymbol{I1XiIXib}{0}{\draw  (0,-0.25) node[xi] {} -- (0,1) node[not] {};
\draw[kernels2] (0,1) -- (0,2.25) ; \draw (0,2.25) node[xi]{}; }

\DeclareSymbol{I1XiIXic}{0}{
\draw[kernels2] (0,0) -- (1,1) node[xi] {} ; 
\draw[kernels2] (0,0) node[not] {}  -- (-1,1) node[not] {} -- (0,2) node[xi] {};
}

\DeclareSymbol{thin}{1.4}{\draw[pagebackground] (-0.3,0) -- (0.3,0); \draw  (0,0) -- (0,2);}
\DeclareSymbol{thin2}{1.4}{\draw[pagebackground] (-0.3,0) -- (0.3,0); \draw[tinydots]  (0,0) -- (0,2);}

\DeclareSymbol{thick}{1.4}{\draw[pagebackground] (-0.3,0) -- (0.3,0); \draw[kernels2]  (0,0) -- (0,2);}
\DeclareSymbol{thick2}{1.4}{\draw[pagebackground] (-0.3,0) -- (0.3,0); \draw[kernels2,tinydots]  (0,0) -- (0,2);}

\DeclareSymbol{Xi4ind}{2}{\draw (0,0) node[xi,label={[label distance=-0.2em]right: \scriptsize  $ i $}]  { } -- (-1,1) node[xi,label={[label distance=-0.2em]left: \scriptsize  $ j $}] {} -- (0,2) node[xi,label={[label distance=-0.2em]right: \scriptsize  $ k $}] {} -- (-1,3) node[xi,label={[label distance=-0.2em]left: \scriptsize  $ \ell $}] {};}

\DeclareSymbol{Xi4c1}{2}{\draw (0,0) node[xic] {} -- (-1,1) node[xi] {} -- (0,2) node[xic] {} -- (-1,3) node[xi] {};} 
\DeclareSymbol{IXi2ex}{0}{\draw (0,-0.25) node[xie] {} -- (-1,1) node[xi] {} ; \draw (0,-0.25)-- (1,1) node[xi] {};}
\DeclareSymbol{IXi2ex1}{0}{\draw (0,-0.25) node[xie] {} -- (-1,1) node[xi] {} -- (0,2) node[xi] {};}

\DeclareSymbol{Xi4b1}{0}{\draw(0,1.5) node[xic] {} -- (0,0); \draw (-1,1) node[xic] {} -- (0,0) node[xi] {} -- (1,1) node[xi] {};}

\DeclareSymbol{Xi4ec1}{0}{\draw (0,2) node[xi] {} -- (-1,1) node[xic] {} -- (0,0) node[xic] {} -- (1,1) node[xi] {};}
\DeclareSymbol{Xi4ec2}{0}{\draw (0,2) node[xic] {} -- (-1,1) node[xi] {} -- (0,0) node[xic] {} -- (1,1) node[xi] {};}
\DeclareSymbol{Xi4ec3}{0}{\draw (0,2) node[xic] {} -- (-1,1) node[xic] {} -- (0,0) node[xi] {} -- (1,1) node[xi] {};}

\DeclareSymbol{I1Xi4ac1}{2}{\draw[kernels2] (0,0) node[not] {} -- (-1,1) ; \draw[kernels2] (0,0) node[not] {} -- (1,1) node[xic] {} ;
\draw (-1,1) node[xi] {} -- (0,2) node[xic] {} -- (-1,3) node[xi] {};}

\DeclareSymbol{I1Xi4ac2}{2}{\draw[kernels2] (0,0) node[not] {} -- (-1,1) ; \draw[kernels2] (0,0) node[not] {} -- (1,1) node[xic] {} ;
\draw (-1,1) node[xi] {} -- (0,2) node[xi] {} -- (-1,3) node[xic] {};}

\DeclareSymbol{I1Xi4bp}{2}{\draw (0,0) node[not] {} -- (-1,1) node[xi] {} -- (0,2) ; \draw[kernels2] (0,2) node[not] {} -- (-1,3) node[xi] {};\draw[kernels2] (0,2)  -- (1,3) node[xi] {};
}

\DeclareSymbol{I1Xi4bc1}{2}{\draw (0,0) node[xic] {} -- (-1,1) node[xi] {} -- (0,2) ; \draw[kernels2] (0,2) node[not] {} -- (-1,3) node[xi] {};\draw[kernels2] (0,2)  -- (1,3) node[xic] {};
}

\DeclareSymbol{I1Xi4bc2}{2}{\draw (0,0) node[xic] {} -- (-1,1) node[xi] {} -- (0,2) ; \draw[kernels2] (0,2) node[not] {} -- (-1,3) node[xic] {};\draw[kernels2] (0,2)  -- (1,3) node[xi] {};
}

\DeclareSymbol{I1Xi4cp}{2}{\draw (0,0) node[not] {} -- (-1,1) node[not] {}; \draw[kernels2] (-1,1) -- (0,2) ; 
\draw[kernels2] (-1,1) -- (-2,2) node[xi] {} ;
\draw (0,2) node[xi] {} -- (-1,3) node[xi] {};}

\DeclareSymbol{I1Xi4cc1}{2}{\draw (0,0) node[xic] {} -- (-1,1) node[not] {}; \draw[kernels2] (-1,1) -- (0,2) ; 
\draw[kernels2] (-1,1) -- (-2,2) node[xi] {} ;
\draw (0,2) node[xic] {} -- (-1,3) node[xi] {};}

\DeclareSymbol{I1Xi4cc2}{2}{\draw (0,0) node[xic] {} -- (-1,1) node[not] {}; \draw[kernels2] (-1,1) -- (0,2) ; 
\draw[kernels2] (-1,1) -- (-2,2) node[xi] {} ;
\draw (0,2) node[xi] {} -- (-1,3) node[xic] {};}

\DeclareSymbol{I1Xi4abc1}{2}{\draw[kernels2] (0,0) node[not] {} -- (-1,1) ; \draw[kernels2] (0,0) node[not] {} -- (1,1) node[xic] {};\draw (-1,1) node[xi] {} -- (0,2) ; \draw[kernels2] (0,2) node[not] {} -- (-1,3) node[xic] {};\draw[kernels2] (0,2)  -- (1,3) node[xi] {}; }

\DeclareSymbol{I1Xi4abc2}{2}{\draw[kernels2] (0,0) node[not] {} -- (-1,1) ; \draw[kernels2] (0,0) node[not] {} -- (1,1) node[xic] {};\draw (-1,1) node[xi] {} -- (0,2) ; \draw[kernels2] (0,2) node[not] {} -- (-1,3) node[xi] {};\draw[kernels2] (0,2)  -- (1,3) node[xic] {}; }

\DeclareSymbol{R1}{0}{\draw (-1,1) node[xi] {} -- (0,0) node[not] {};
\draw[kernels2] (0,1.5) node[xic] {} -- (0,0) -- (1,1) node[xic] {};}
\DeclareSymbol{R2}{0}{\draw (-1,1) node[xic] {} -- (0,0) node[not] {};
\draw[kernels2] (0,1.5)  {} -- (0,0) -- (1,1)  {};
\draw (0,1.5) node[xi] {};
\draw (1,1) node[xic] {};
}
\DeclareSymbol{R3}{1}{\draw[kernels2] (-1,1.5)  {} -- (0,0) node[not] {} -- (1,1.5);
\draw (-1,1.5) node[xi] {};
\draw[kernels2] (0,3) {} -- (1,1.5) -- (2,3)  {};
\draw  (0,3) node[xic] {} ;
\draw (2,3) node[xic] {};}
\DeclareSymbol{R4}{1}{\draw[kernels2] (-1,1.5) node[xic] {} -- (0,0) node[not] {} -- (1,1.5);
\draw[kernels2] (0,3) {} -- (1,1.5) -- (2,3) node[xic] {};
\draw (0,3) node[xi] {};}

\DeclareSymbol{I1Xi4bcp}{2}{\draw (0,0) node[not] {} -- (-1,1) node[not] {}; \draw[kernels2] (-1,1) -- (0,2) ; 
\draw[kernels2] (-1,1) -- (-2,2) node[xi] {} ; \draw[kernels2] (0,2) node[not] {} -- (-1,3) node[xi] {};\draw[kernels2] (0,2)  -- (1,3) node[xi] {};
}

\DeclareSymbol{I1Xi4bcc1}{2}{\draw (0,0) node[xic] {} -- (-1,1) node[not] {}; \draw[kernels2] (-1,1) -- (0,2) ; 
\draw[kernels2] (-1,1) -- (-2,2) node[xi] {} ; \draw[kernels2] (0,2) node[not] {} -- (-1,3) node[xi] {};\draw[kernels2] (0,2)  -- (1,3) node[xic] {};
}

\DeclareSymbol{I1Xi4bcc2}{2}{\draw (0,0) node[xic] {} -- (-1,1) node[not] {}; \draw[kernels2] (-1,1) -- (0,2) ; 
\draw[kernels2] (-1,1) -- (-2,2) node[xi] {} ; \draw[kernels2] (0,2) node[not] {} -- (-1,3) node[xic] {};\draw[kernels2] (0,2)  -- (1,3) node[xi] {};
} 

\DeclareSymbol{2I1Xi4bc1}{2}{\draw[kernels2] (0,0) node[not] {} -- (-1,1) ;
\draw[kernels2] (0,0) -- (1,1);
\draw (-1,1) node[xic] {} -- (-1,2.5) node[xi] {};
\draw (1,1)  node[xic] {} -- (1,2.5) node[xi] {};
}

\DeclareSymbol{2I1Xi4bc2}{2}{\draw[kernels2] (0,0) node[not] {} -- (-1,1) ;
\draw[kernels2] (0,0) -- (1,1);
\draw (-1,1) node[xi] {} -- (-1,2.5) node[xic] {};
\draw (1,1)  node[xic] {} -- (1,2.5) node[xi] {};
}

\DeclareSymbol{diff2I1Xi4bc2}{2}{\draw (0,0) node[diff] {} -- (-1,1) ;
\draw (0,0) -- (1,1);
\draw (-1,1) node[xi] {} -- (-1,2.5) node[xic] {};
\draw (1,1)  node[xic] {} -- (1,2.5) node[xi] {};
}

\DeclareSymbol{2I1Xi4bc3}{2}{\draw[kernels2] (0,0) node[not] {} -- (-1,1) ;
\draw[kernels2] (0,0) -- (1,1);
\draw (-1,1) node[xic] {} -- (-1,2.5) node[xic] {};
\draw (1,1)  node[xi] {} -- (1,2.5) node[xi] {};
}

\DeclareSymbol{Xi41}{0}{\draw (0,1) -- (0.8,2.2) node[xic] {};\draw (0,-0.25) node[xi] {} -- (0,1) node[xi] {} -- (-0.8,2.2) node[xic] {};} 

\DeclareSymbol{Xi42}{0}{\draw (0,1) -- (0.8,2.2) node[xi] {};\draw (0,-0.25) node[xic] {} -- (0,1) node[xi] {} -- (-0.8,2.2) node[xic] {};}

\DeclareSymbol{Xi4ca1}{0}{\draw (0,1) -- (-1,2.2) node[xic] {};\draw (0,-0.25) node[xi] {} -- (0,1) ; \draw[kernels2] (0,1) node[not] {} -- (1,2.2) node[xic] {};
\draw[kernels2] (0,1) {} -- (0,2.7) node[xi] {};
}

\DeclareSymbol{Xi4ca2}{0}{\draw (0,1) -- (-1,2.2) node[xi] {};\draw (0,-0.25) node[xi] {} -- (0,1) ; \draw[kernels2] (0,1) node[not] {} -- (1,2.2) node[xic] {};
\draw[kernels2] (0,1) {} -- (0,2.7) node[xic] {};
}

\DeclareSymbol{Xi4cap}{0}{\draw (0,1) -- (-1,2.2) node[xi] {};\draw (0,-0.25) node[not] {} -- (0,1) ; \draw[kernels2] (0,1) node[not] {} -- (1,2.2) node[xi] {};
\draw[kernels2] (0,1) {} -- (0,2.7) node[xi] {};
}

\DeclareSymbol{Xi3a}{0}{
 \draw (-1,1)  node[xi] {} -- (0,0); 
 \draw (0,0) node[xi] {}  -- (1,1) node[xi] {};
}

\DeclareSymbol{Xi4ebc1}{0}{
\draw[kernels2] (0,2) node[xi] {} -- (-1,1) ; \draw[kernels2] (-2,2)  node[xic] {} -- (-1,1) ; \draw (-1,1)  node[not] {} -- (0,0); 
 \draw (0,0) node[xic] {}  -- (1,1) node[xi] {};
}

\DeclareSymbol{Xi4ebc2}{0}{
\draw[kernels2] (0,2) node[xi] {} -- (-1,1) ; \draw[kernels2] (-2,2)  node[xi] {} -- (-1,1) ; \draw (-1,1)  node[not] {} -- (0,0); 
 \draw (0,0) node[xic] {}  -- (1,1) node[xic] {};
}

\DeclareSymbol{Xi2cbispex}{0}{\draw[kernels2] (0,1) -- (0.8,2.2) node[xi] {};\draw (0,-0.25) node[xie] {} -- (0,1); \draw[kernels2] (0,1) node[not] {} -- (-0.8,2.2) node[xi] {};}

\DeclareSymbol{Xi2cbis1p}{0}{\draw (0,1) -- (-0.8,2.2) node[xi] {};\draw (0,-0.25) node[not] {} -- (0,1) node[xi] {}; }

\DeclareSymbol{Xi2Xp}{-2}{\draw (0,-0.25) node[not] {} -- (-1,1) node[xix] {};} 

\DeclareSymbol{I1XiIXib}{0}{\draw  (0,-0.25) node[xi] {} -- (0,1) node[not] {};
\draw[kernels2] (0,1) -- (0,2.25) ; \draw (0,2.25) node[xi]{}; }

\DeclareSymbol{IXi2b}{0}{\draw  (0,-0.25) node[xi] {} -- (0,1) node[not] {};
\draw (0,1) -- (0,2.25) ; \draw (0,2.25) node[xi]{}; }

\DeclareSymbol{IXi2bex}{0}{\draw  (0,-0.25) node[xi] {} -- (0,1) node[xie] {};
\draw (0,1) -- (0,2.25) ; \draw (0,2.25) node[xi]{}; }

 \def\1{\mathbf{\symbol{1}}}

\def\one{\mathbf{1}}

\DeclareSymbol{diff}{0}{
\draw (0,0.5) node[diff] {};
}

\DeclareSymbol{diff1}{0}{
\draw (0,0.5) node[diff1] {};
}

\DeclareSymbol{diff2}{0}{
\draw (0,0.5) node[diff2] {};
}

\DeclareSymbol{geo}{0}{
\draw (0,0) node[diff] {};
\draw (0.3,0) node[diff] {};
}

\DeclareSymbol{generic}{0}{
\draw (0,0.6) node[xi] {};
}

\DeclareSymbol{g}{0}{
\draw (0,0.6) node[g] {};
}

\DeclareSymbol{Ito}{0}{
\draw (0,0.6) node[xies] {};
}

\DeclareSymbol{Itob}{0}{
\draw (0,0.6) node[xiesf] {};
}

\DeclareSymbol{greycirc}{0}{
\draw (0,0.3) node[xi] {};
}

\DeclareSymbol{not}{0}{
\draw (0,0.6) node[not] {};
\draw[tinydots] (0,0.6) circle (0.8);
}

\DeclareSymbol{genericb}{0}{
\draw (0,0.6) node[xic] {};
}

\DeclareSymbol{bluecirc}{0}{
\draw (0,0.3) node[xic] {};
}

\DeclareSymbol{genericxix}{0}{
\draw (0,0.6) node[xix] {};
}

\DeclareSymbol{genericX}{0}{
\draw (0,0.6) node[X] {};
}

\DeclareSymbol{diffIto}{1}{
\draw  (0,2.5) -- (0,0) ;
\draw (0,-0.1) node[diff] {};
\draw (0,2.5) node[xies] {};
}
\DeclareSymbol{Itodiff}{2}{
\draw(0,2.9) -- (0,-0.2);
\draw (0,2.9) node[diff] {};
\draw (0,-0.1) node[xies] {};
}

\DeclareSymbol{diffgeneric}{1}{
\draw  (0,2.5) -- (0,0) ;
\draw (0,-0.1) node[diff] {};
\draw (0,2.5) node[xi] {};
}

\DeclareSymbol{genericdiff}{2}{
\draw(0,2.9) -- (0,-0.2);
\draw (0,2.9) node[diff] {};
\draw (0,-0.1) node[xi] {};
}

\DeclareSymbol{diffdot}{2}{
\draw  (0,3) -- (0,-0.1) ;
\draw (0,3) node[not] {};
\draw (0,-0.1) node[diff] {};
}

\DeclareSymbol{diffdotmini}{0}{
\draw  (0,0) -- (0,1.2) ;
\draw (0,1.2) node[not] {};
\draw (0,0) node[diffmini] {};
}

\DeclareSymbol{dotdiff}{2}{
\draw[kernelsmod]  (0,3) -- (0,-0.1) ;
\draw (0,3) node[diff] {};
\draw (0,-0.1) node[not] {};
}

\DeclareSymbol{dotdiff1}{2}{
\draw[kernelsmod]  (0,3) -- (0,-0.1) ;
\draw (0,3) node[diff1] {};
\draw (0,-0.1) node[not] {};
}

\DeclareSymbol{dotdiff1mini}{0}{
\draw[kernelsmod]  (0,1.2) -- (0,0) ;
\draw (0,1.2) node[diffmini] {};
\draw (0,0) node[not] {};
}

\DeclareSymbol{dotdiff2}{2}{
\draw (0,3) -- (0,-0.1) ;
\draw (0,3) node[diff] {};
\draw (0,-0.1) node[not] {};
}

\DeclareSymbol{dotdiff2mini}{0}{
\draw (0,1.2) -- (0,0) ;
\draw (0,1.2) node[diffmini] {};
\draw (0,0) node[not] {};
}

\DeclareSymbol{dotdiffstraight}{0}{
\draw  (0,3) -- (0,-0.1) ;
\draw (0,3) node[diff] {};
\draw (0,-0.1) node[not] {};
}

\DeclareSymbol{arbre1}{0}{
\draw  (0,0) -- (1.5,1.5) ;
\draw (1.5,1.5) node[not] {};
\draw (0,0) node[not] {};
}

\DeclareSymbol{arbre2}{0}{
\draw  (0,0) -- (1.5,1.5) ;
\draw[kernelsmod] (0,0) -- (-1.5,1.5);
\draw (1.5,1.5) node[not] {};
\draw (0,0) node[not] {};
\draw (-1.5,1.5) node[xi] {};
}

\DeclareSymbol{arbre3}{0}{
\draw  (0,0) -- (1.5,1.5) ;
\draw[kernelsmod] (1.5,1.5) -- (0,3);
\draw (0,0) node[not] {};
\draw (1.5,1.5) node[not] {};
\draw (0,3) node[xi] {};
}

\DeclareSymbol{treeeval}{0}{
\draw (0,0) -- (1,1);
\draw (0,0) node[xi] {};
\draw (1.25,1.25) node[xi] {};
\draw (-0.6,0.6) node[]{\tiny{$i$}};
\draw (0.65,1.85) node[]{\tiny{$j$}};
}

\DeclareSymbol{testeval}{0}{
\draw (0,0) -- (1,1);
\draw (0,0) -- (-1,1);
\draw (0,0) node[xi] {};
\draw (1.25,1.25) node[xi] {};
\draw (-1.25,1.25) node[xi] {};
\draw (-0.6,-0.6) node[]{\tiny{$i$}};
\draw (0.65,1.85) node[]{\tiny{$j$}};
\draw (-1.95,1.85) node[]{\tiny{$k$}};
}

\DeclareSymbol{treeeval2}{0}{
\draw[kernelsmod] (-0.25,-1) -- (1,0.5) ;
\draw[kernelsmod] (1,0.5) -- (-0.25,2);
\draw (1,0.5) node[diff2] {};
\draw (-0.25,-1) node[not] {};
\draw (-0.25,2) node[xi] {};
\draw (-0.6,1.2) node[]{\tiny{1}};
}

\DeclareSymbol{arbreact}{1}{
\draw (0,0) node[not] {};
\draw[kernelsmod] (0,0) -- (1,1);
\draw[kernelsmod] (0,0) -- (-1,1);
\draw (-1,1) node[xic] {};
\draw  (0,2) -- (1,1) ;
\draw (0,2) node[xic] {};
\draw (1,1) node[xi] {};
}

\DeclareSymbol{arbreact1}{0}{
\draw (0,-1.5) -- (0,0);
\draw[kernelsmod] (0,0) -- (1,1);
\draw[kernelsmod] (0,0) -- (-1,1);
\draw  (0,2) -- (1,1) ;
\draw (0,-1.5) node[diff] {};
\draw (0,0) node[not] {};
\draw (-1,1) node[xic] {};
\draw (0,2) node[xic] {};
\draw (1,1) node[xi] {};
}

\DeclareSymbol{arbreact2}{0}{
\draw (0,-0.75) -- (-1,0.5); 
\draw (0,-0.75) -- (1,0.5);
\draw (0,1.5) -- (1,0.5);
\draw (0,1.5) node[xic] {};
\draw (1,0.5) node[xi] {};
\draw (-1,0.5) node[xic] {};
\draw (0,-0.75) node[diff] {};
}

\DeclareSymbol{arbreact3}{0}{
\draw[kernelsmod] (0,-0.75) -- (-1,0.5); 
\draw[kernelsmod] (0,-0.75) -- (1,0.5);
\draw (0,1.75) -- (1,0.5);
\draw (2,1.75) -- (1,0.5);
\draw (0,1.75) node[xic] {};
\draw (1,0.5) node[diff] {};
\draw (-1,0.5) node[xic] {};
\draw (2,1.75) node[xi] {};
\draw (0,-0.75) node[not] {};
}

\DeclareSymbol{pre_im_1}{0}{
\draw[kernels2] (0,-0.5) node[not] {} -- (-0.6,0.5) ;
\draw[kernels2] (0,-0.5) -- (0.6,0.5);
\draw (0,1.1)  -- (-0.55,2);
\draw (0,1.1)  -- (0.55,2);
\draw (0,0.7) node[g] {};
\draw (0,2.2) node[g] {};
}

\DeclareSymbol{disconnect}{0}{
\draw[kernels2] (0,-0.5) node[not] {} -- (-0.6,0.5) ;
\draw[kernels2] (0,-0.5) -- (0.6,0.5);
\draw (-0.55,1.1)  -- (-0.55,2.3);
\draw (0.55,2.3) -- (0.55,1.5) -- (1.2,1.5) -- (1.2,3.5) -- (0.55,3.5) -- (0.55,2.7);
\draw (0,0.7) node[g] {};
\draw (0,2.5) node[g] {};
}

\DeclareSymbol{pre_im_2}{2}{\draw[kernels2] (0,0) node[not] {} -- (-1,1) node[not] {};
\draw[kernels2] (0,0) -- (1,1) node[not] {};
\draw[kernels2] (-1,1) -- (-1.5,2.5);
\draw[kernels2] (-1,1) -- (-0.5,2.5);
\draw[kernels2] (1,1) -- (0.5,2.5);
\draw[kernels2] (1,1) -- (1.5,2.5);
\draw (-1,2.7) node[g] {};
\draw (1,2.7) node[g] {};
}

\DeclareSymbol{CX_rec}{0}{
\draw [black] (-0.3,1) to (-0.3,-0.3);
\draw [black] (0.3,1) to (0.3,-0.3);
\draw [black] (-0.3,1) to (-0.3,2.3);
\draw [black] (0.3,1) to (0.3,2.3);
\draw (0,1) node[rec] {};
}

\DeclareSymbol{CX_cerc}{0}{
\draw [black] (0,1) to (0,-0.3);
\draw (0,1) node[cerc] {};
}

\DeclareSymbol{proof0}{0}{
\draw (0,-3) node[] {\tiny$\tau_0$};
\draw (0,-2.3) -- (0,0.5);`
\draw (0,0.5) -- (-1.5,2.5);
\draw (0,0.5) -- (1.5,2.5);
\draw (-1,-2) node[] {\tiny$v$};
\draw (-1.5,3.1) node[] {\tiny$\tau_1$};
\draw (1.5,3.1) node[] {\tiny$\tau_2$};
\draw (0,0.5) node[diff] {};
}

\DeclareSymbol{proof0b}{0}{
\draw (0,0.5) -- (-1.5,2.5);
\draw (0,0.5) -- (1.5,2.5);
\draw (-1.5,3.1) node[] {\tiny$\tau_1$};
\draw (1.5,3.1) node[] {\tiny$\tau_2$};
\draw (0,0.5) node[diff] {};
}

\DeclareSymbol{proof}{0}{
\draw[kernelsmod] (-2,3) -- (0,0);
\draw[kernelsmod] (2,3) -- (0,0);
\draw (0,0) node[not] {};
}

\DeclareSymbol{prooftri}{0}{
\draw[kernelsmod] (-2,3) -- (0,0);
\draw[kernelsmod] (0,4) -- (0,0);
\draw (2,2.7)--(0,0);
\draw (0,0) node[not] {};
}

\DeclareSymbol{proofqua}{0}{
\draw[kernelsmod] (-3,3) -- (0,0);
\draw[kernelsmod] (-1,4) -- (0,0);
\draw (1,3.6)--(0,0);
\draw (3,2.7)--(0,0);
\draw (0,0) node[not] {};
}

\DeclareSymbol{proof1_1}{0}{
\draw (0,-2.7) node[] {\tiny$\tau_0$};
\draw (0,-2) -- (0,0.5);`
\draw[kernelsmod] (0,0.5) -- (-1.5,2.5); 
\draw[kernelsmod] (0,0.5) -- (1.5,2.5);
\draw (-1,-1.7) node[] {\tiny$v$};
\draw (-1.5,3.1) node[] {\tiny$\tau_1$};
\draw (1.5,3.1) node[] {\tiny$\tau_2$};
\draw (0,0.5) node[not] {};
}

\DeclareSymbol{proof1b_1}{0}{
\draw[kernelsmod] (0,0.5) -- (-1.5,2.5); 
\draw[kernelsmod] (0,0.5) -- (1.5,2.5);
\draw (-1.5,3.1) node[] {\tiny$\tau_1$};
\draw (1.5,3.1) node[] {\tiny$\tau_2$};
\draw (0,0.5) node[not] {};
}

\DeclareSymbol{proof1_2}{0}{
\draw (0,-2.7) node[] {\tiny$\tau_0$};
\draw[kernelsmod] (0,-2) -- (0,0.5);`
\draw[kernelsmod] (0,0.5) -- (-1.5,2.5); 
\draw[kernelsmod] (0,0.5) -- (1.5,2.5);
\draw (-1,-1.7) node[] {\tiny$v$};
\draw (-1.5,3.1) node[] {\tiny$\tau_1$};
\draw (1.5,3.1) node[] {\tiny$\tau_2$};
\draw (0,0.5) node[not] {};
}

\DeclareSymbol{proof2_1}{0}{
\draw (0,-2.7) node[] {\tiny$\tau_0$};
\draw (0,-2) -- (-1.8,-0.3);
\draw (-1.8,0.7) -- (0,2.5); 
\draw (1,-1.7) node[] {\tiny$v$};
\draw (-2.5,1.2) node[] {\tiny$r_1$};
\draw (0,3.1) node[] {\tiny$\tau_2$};
\draw (-2.5,0) node[] {\tiny$\tau_1$};
}

\DeclareSymbol{proof2b_1}{0}{
\draw (-1.8,0.7) -- (0,2.5); 
\draw (-2.5,1.2) node[] {\tiny$r_1$};
\draw (0,3.1) node[] {\tiny$\tau_2$};
\draw (-2.5,0) node[] {\tiny$\tau_1$};
}

\DeclareSymbol{proof2b_2}{0}{
\draw (-1.8,0.7) -- (0,2.5); 
\draw (-2.5,1.2) node[] {\tiny$r_2$};
\draw (0,3.1) node[] {\tiny$\tau_1$};
\draw (-2.5,0) node[] {\tiny$\tau_2$};
}

\DeclareSymbol{proof2_2}{0}{
\draw (0,-2.7) node[] {\tiny$\tau_0$};
\draw[kernelsmod] (0,-2) -- (-1.8,-0.3);
\draw (-1.8,0.7) -- (0,2.5); 
\draw (1,-1.7) node[] {\tiny$v$};
\draw (-2.5,1.2) node[] {\tiny$r_1$};
\draw (0,3.1) node[] {\tiny$\tau_2$};
\draw (-2.5,0) node[] {\tiny$\tau_1$};
}

\DeclareSymbol{proof3_1}{0}{
\draw (0,-2.7) node[] {\tiny$\tau_0$};
\draw (0,-2) -- (1.8,-0.3);
\draw (1.8,0.7) -- (0,2.5); 
\draw (-1,-1.7) node[] {\tiny$v$};
\draw (2.5,1.2) node[] {\tiny$r_2$};
\draw (0,3.1) node[] {\tiny$\tau_1$};
\draw (2.5,0) node[] {\tiny$\tau_2$};
}

\DeclareSymbol{proof3_2}{0}{
\draw (0,-2.7) node[] {\tiny$\tau_0$};
\draw[kernelsmod] (0,-2) -- (1.8,-0.3);
\draw (1.8,0.7) -- (0,2.5); 
\draw (-1,-1.7) node[] {\tiny$v$};
\draw (2.5,1.2) node[] {\tiny$r_2$};
\draw (0,3.1) node[] {\tiny$\tau_1$};
\draw (2.5,0) node[] {\tiny$\tau_2$};
}

\DeclareSymbol{proof4_1}{0}{
\draw (0,-2) node[] {\tiny$\tau_0$};
\draw (0,-1.3) -- (-1.5,1.8); 
\draw  (0,-1.3) -- (1.5,1.8);
\draw (-1,-1) node[] {\tiny$v$};
\draw (-1.5,2.4) node[] {\tiny$\tau_1$};
\draw (1.5,2.4) node[] {\tiny$\tau_2$};
}

\DeclareSymbol{proof4_2}{0}{
\draw (0,-2) node[] {\tiny$\tau_0$};
\draw[kernelsmod] (0,-1.3) -- (-1.5,1.8); 
\draw  (0,-1.3) -- (1.5,1.8);
\draw (-1,-1) node[] {\tiny$v$};
\draw (-1.5,2.4) node[] {\tiny$\tau_1$};
\draw (1.5,2.4) node[] {\tiny$\tau_2$};
}

\DeclareSymbol{proof4_3}{0}{
\draw (0,-2) node[] {\tiny$\tau_0$};
\draw (0,-1.3) -- (-1.5,1.8); 
\draw[kernelsmod]  (0,-1.3) -- (1.5,1.8);
\draw (-1,-1) node[] {\tiny$v$};
\draw (-1.5,2.4) node[] {\tiny$\tau_1$};
\draw (1.5,2.4) node[] {\tiny$\tau_2$};
}

\DeclareSymbol{prooftriple}{0}{
\draw (0,-2.7) node[] {\tiny$\tau_0$};
\draw (0,-2) -- (0,0.25);`
\draw (0,0.25) -- (-1.5,2.5); 
\draw (0,0.25) -- (1.5,2.5);
\draw (0,0.25) -- (0,2.5);
\draw (-1,-1.7) node[] {\tiny$v$};
\draw (-2,3.1) node[] {\tiny$\tau_1$};
\draw (0,3.1) node[] {\tiny$\tau_2$};
\draw (2,3.1) node[] {\tiny$\tau_3$};
\draw (0,0.25) node[diff] {};
}

\DeclareSymbol{prooftriple1}{0}{
\draw (0,-2.7) node[] {\tiny$\tau_0$};
\draw (0,-2) -- (0,0.25);
\draw[kernelsmod] (0,0.25) -- (-1.5,2.5); 
\draw (0,0.25) -- (1.5,2.5);
\draw[kernelsmod] (0,0.25) -- (0,2.5);
\draw (-1,-1.7) node[] {\tiny$v$};
\draw (-2,3.1) node[] {\tiny$\tau_1$};
\draw (0,3.1) node[] {\tiny$\tau_2$};
\draw (2,3.1) node[] {\tiny$\tau_3$};
\draw (0,0.25) node[not] {};
}

\DeclareSymbol{prooftripleperm1}{0}{
\draw (0,-2.7) node[] {\tiny$\tau_0$};
\draw (0,-2) -- (0,0.25);
\draw[kernelsmod] (0,0.25) -- (-3.5,2.5); 
\draw (0,0.25) -- (2.5,2.5);
\draw[kernelsmod] (0,0.25) -- (0,2.5);
\draw (-1,-1.7) node[] {\tiny$v$};
\draw (-3,3.1) node[] {\tiny$\tau_{\sigma_1}$};
\draw (0,3.1) node[] {\tiny$\tau_{\sigma_2}$};
\draw (3,3.1) node[] {\tiny$\tau_{\sigma_3}$};
\draw (0,0.25) node[not] {};
}

\DeclareSymbol{proofdouble}{0}{
\draw (0,0.5) -- (-1.5,2.5);
\draw (0,0.5) -- (1.5,2.5);
\draw (-1.5,3.1) node[] {\tiny$\tau_1$};
\draw (1.5,3.1) node[] {\tiny$\tau_2$};
\draw (0,0.5) node[diff] {};
}

\DeclareSymbol{proofquadruple}{0}{
\draw (0,0) -- (-2.5,2.5); 
\draw (0,0) -- (-1,2.5);
\draw (0,0) -- (1,2.5);
\draw (0,0) -- (2.5,2.5);
\draw (-3,3.1) node[] {\tiny$\tau_1$};
\draw (-1,3.1) node[] {\tiny$\tau_2$};
\draw (1,3.1) node[] {\tiny$\tau_3$};
\draw (3,3.1) node[] {\tiny$\tau_4$};
\draw (0,0) node[diff] {};
}

\DeclareSymbol{proofquadruple1}{0}{
\draw[kernelsmod] (0,0) -- (-2.5,2.5); 
\draw[kernelsmod] (0,0) -- (-1,2.5);
\draw (0,0) -- (1,2.5);
\draw (0,0) -- (2.5,2.5);
\draw (-3,3.1) node[] {\tiny$\tau_1$};
\draw (-1,3.1) node[] {\tiny$\tau_2$};
\draw (1,3.1) node[] {\tiny$\tau_3$};
\draw (3,3.1) node[] {\tiny$\tau_4$};
\draw (0,0) node[not] {};
}

\DeclareSymbol{proofquadrupleperm1}{0}{
\draw[kernelsmod] (0,-0.4) -- (-3.5,2.3); 
\draw[kernelsmod] (0,-0.4) -- (-1,2.3);
\draw (0,-0.4) -- (1,2.5);
\draw (0,-0.4) -- (3.5,2.5);
\draw (-4.5,3.1) node[] {\tiny$\tau_{\sigma_1}$};
\draw (-1.5,3.1) node[] {\tiny$\tau_{\sigma_2}$};
\draw (1.5,3.1) node[] {\tiny$\tau_{\sigma_3}$};
\draw (4.5,3.1) node[] {\tiny$\tau_{\sigma_4}$};
\draw (0,-0.4) node[not] {};
}

\DeclareMathAlphabet{\mathpzc}{OT1}{pzc}{m}{it}

\def\eqref#1{(\ref{#1})}

\makeatletter 
\newcommand*{\bigcdot}{}
\DeclareRobustCommand*{\bigcdot}{%
  \mathbin{\mathpalette\bigcdot@{}}%
}
\newcommand*{\bigcdot@scalefactor}{.5}
\newcommand*{\bigcdot@widthfactor}{1.15}
\newcommand*{\bigcdot@}[2]{%
  \sbox0{$#1\vcenter{}$}
  \sbox2{$#1\cdot\m@th$}%
  \hbox to \bigcdot@widthfactor\wd2{%
    \hfil
    \raise\ht0\hbox{%
      \scalebox{\bigcdot@scalefactor}{%
        \lower\ht0\hbox{$#1\bullet\m@th$}%
      }%
    }%
    \hfil
  }%
}
\makeatother

\tcbset
{colframe=boxcolor,colback=symbols!7!pagebackground,coltext=pageforeground,
fonttitle=\bfseries,nobeforeafter,center title,size=fbox,boxsep=1.5pt,
top=0mm,bottom=0mm,boxsep=0mm,tcbox raise base}

\def\two{{\<generic>\kern0.05em\<genericb>}}
\def\twoI{{\<Ito>\kern0.05em\<Itob>}}

\def\mail#1{\burlalt{#1}{mailto:#1}}

\begin{document}

\title{Renormalisation from non-geometric to geometric rough paths}
\author{Yvain Bruned}
\institute{University of Edinburgh,
\, 
\mail{Yvain.Bruned@ed.ac.uk}.
}

\maketitle

\begin{abstract}
The Hairer-Kelly map has been introduced for establishing a correspondence between geometric and non-geometric rough paths.
Recently, a new renormalisation on rough paths has been proposed in \cite{TZ},  built on this map and the Lyons-Victoir extension theorem. In this work, we compare this renormalisation with the existing ones such as BPHZ  and the local products renormalisations. We prove that they commute in a certain sense with the Hairer-Kelly map and exhibit an explicit formula in the framework of \cite{TZ}. We also see how the renormalisation behaves in the alternative approach in \cite{BC2019} for moving from non-geometric to geometric rough paths.

\end{abstract}

\setcounter{tocdepth}{1}
\tableofcontents

\section{Introduction}

Renormalisation plays a central role in the theory of  singular stochastic partial differential equations (SPDEs). Since the foundation paper \cite{reg} of Martin Hairer establishing the theory of Regularity Structures, it has been understood that the mains objects describing the solution of a singular SPDE need to be renormalised and that they reflect the ill-defined distributional products appearing in the equation. The solution is described locally via recentered iterated integrals called model.
This representation is directly inspired by rough paths \cite{Lyons98} and controlled rough paths \cite{Gubinelli2004}. The iterated integrals are constructed from characters on a combinatorial Hopf algebra. The algebra at play in regularity structures is the one of decorated rooted trees 
which is close to the branched rough paths introduced in \cite{Gub06}. Equipped with a deformation of a Butcher-Connes-Kreimer  coproduct \cite{Butcher,CK}, it forms a Hopf algebra. 
The renormalisation chosen for the model is the BPHZ renormalisation coming from perturbative quantum field theory \cite{BP,Hepp,Zimmermann}. It is implemented using an extraction/contraction Hopf algebra which cointeracts with the Butcher-Connes-Kreimer one. This cointeraction has been first noticed in numerical analysis for B-series \cite{MR2657947} and at the coproducts level in \cite{CA}. Its extension to deformed structures is one of the main results of \cite{BHZ}. The convergence of the renormalised model has been established in \cite{ajay} and 
\cite{BCCH} describes the renormalised equations which allow the resolution of a large class of singular SPDEs. The papers \cite{EMS,BaiHos} give surveys on these developments. For an introduction to Regularity Structures, one can consult the textbook \cite{FrizHai} by Friz and Hairer.

Meanwhile, it has been noticed that the BPHZ renormalisation has a nice counterpart in rough paths theory \cite{BCFP}. Some examples where renormalisation is needed in singular stochastic differential equations (SDEs) are given in \cite{BCF}. The construction of \cite{BCFP} is based on a primal point of view where the renormalisation is viewed as translation maps acting on Lie series.
In \cite{TZ}, the authors proposed a rather new renormalisation which is based on the Hairer-Kelly map introduced in \cite{HK15}. The idea of the latter paper was to fill the gap between two representations for rough paths where various Hopf algebras can be used. The Hairer-Kelly map allows us to move from the Butcher-Connes-Kreimer Hopf algebra (resp. non-geometric rough paths also called branched rough paths) to the tensor Hopf algebra (resp. geometric rough paths).
This change of point of view relies strongly on the Lyons-Victoir extension theorem \cite{LV07} which is a way to lift a path to a rough path. 
In \cite{TZ}, the renormalisation is performed before applying the extension theorem by adding the increment of a Hölder function $ (g^{\tau})_{\tau} $ indexed by the rooted trees used for the branched rough paths. With this parametrisation, they obtain a bijection between branched $ \gamma $-rough paths and anisotropic $ \gamma $-rough paths. They make a link with the renormalisation in \cite{BCFP} by proving a recursive definition for the $ g $. An explicit formula for the $ g $ is missing.

The main contribution of this paper is to make the link between the two renormalisations more precise by providing an explicit formula for the $ g $. We also introduce a new renormalisation on branched rough paths which is inspired from \cite{Br17}. It is based on local products renormalisation and has been recently used  in \cite{CMW19}.
An explicit formula is also given for this renormalisation. We then investigate a different construction from the Hairer-Kelly map given in \cite{BC2019} where the authors proposed a new isomorphism between the two classes of rough paths. This construction is based on an isomorphism given by Chapoton and Foissy \cite{Foi02,cha10}. 
Then, one can bypass the use of the non-canonical construction provided by the Hairer-Kelly map. We see how the renormalisation behaves toward this construction and we obtain nicer formulae in this context.

Let us outline the paper by summarising the content of its sections. In Section~\ref{sec::2}, we introduce the different Hopf algebras and spaces of rough paths considered in this paper:
the branched $ \gamma $-rough paths on the Butcher-Connes-Kreimer Hopf algebra and the  anisotropic $ \gamma $-rough paths on the tensor Hopf algebra. We recall also the main theorems at play for the sequel like the Lyons-Victoir theorem, the Haire-Kelly construction with its central map $ \psi $, the renormalisation given in \cite{TZ} and the isomorphism $ \Psi $ between the two spaces of rough paths given in \cite{BC2019}.
In Section~\ref{sec::3}, we present a general family of renormalisation maps $ M $ satisfying suitable algebraic and analytical properties in order to act on branched rough paths. We see how they commute with the  maps $ \psi $ and $ \Psi $ and show that 
\begin{equs} \label{alg_ident}
\psi(M) = \bar M \psi, \quad \Psi M^{*} = \tilde{M}^{*} \Psi.
\end{equs}
where the map $ M^{*} $ is the adjoint of $ M $. The maps $ \bar M^{*} $ and $ \tilde M^{*} $ can be interpreted as translation maps as in \cite{BCFP} on suitable tensor algebra spaces.
Then, we check that the renormalisation given in \cite{BCFP} and the one inspired from \cite{Br17} enter this framework. At the end, we get the following identity:
\begin{equs} \label{rough_path_ident}
\langle M^{*}  X_{st},\tau \rangle = \langle \bar M^{*}  \bar X_{st},\psi(\tau) \rangle
\end{equs}
where $\bar X $ is the rough path associated to $ X $ in \cite{TZ}.
Moreover, $ \tilde{M}^{*} $ acts on $ \tilde X $ the rough path associated to $ X $ in  \cite{BC2019}.
 The identity \eqref{alg_ident} is an algebraic interaction between the renormalisation map $ M $ and the maps $ \psi $ and $ \Psi $. Then, \eqref{rough_path_ident} shows how this interaction can be viewed at the level of the rough paths.

 Section~\ref{sec::4} contains the main results of the paper namely an explicit expression for the $ g $ given in Theorem~\ref{formula1} for the renormalisation maps $ M $ introduced in Section~\ref{sec::3}
\begin{equs} \label{formu1}
  g_t^\tau-g_s^\tau= \langle \overline{M^{*}  X_{st}},\tau \rangle - \langle \bar X_{st},\tau\rangle.
\end{equs}
  It is unclear if one can go further  and prove the stronger identity
  \begin{equs} \label{formu12}
  g_t^\tau-g_s^\tau= \langle \bar{M}^{*}  \bar{X}_{st},\tau \rangle - \langle \bar X_{st},\tau\rangle.
\end{equs}
This new formula depends whether one can prove that
\begin{equs}
\overline{M^{*}  X_{st}} = \bar{M}^{*} \bar X
\end{equs}
which can be interpreted as showing a commutation between $ M $ and the non-canonical extension provided by the Lyons-Victoir theorem. In contrast, such identity is true and easy to obtain for the approach advocated in \cite{BC2019}. Indeed, one gets from Theorem~\ref{commute_iso}
\begin{equs}
\Psi(M^{\star} X) =  \tilde{M}^{*} \Psi( X).
\end{equs}
We conclude this introduction by saying that such transfer of structures and renormalisation maps must have a counterpart at the level of Regularity Structures  where most of the objects presented here are at play. A different construction of the model based on a different Hopf algebra is missing and could be investigated in the future. Such program has been started in numerical analysis in  \cite{Murua2017} where the authors consider words instead of trees. Regularity structures trees appear in the recent work \cite{BS} for dispersive PDEs. 
Therefore, a tensor structure seems plausible in the context of singular SPDEs.

\section{Rough Paths setting}

\label{sec::2}

In this section, we present the definitions and propositions needed in the sequel. They are mainly extracted from \cite{TZ}.
Let $ \CT $ the set of rooted trees with nodes decorated by $\{0,...,d\}$. 
We grade elements of $ \tau \in \CT $  by the number $|\tau|$ of their nodes and we set
\[
\CT_n:=\{\tau\in\CT: |\tau|\leq n\},
\qquad n\in\N.
\]
We denote by 
 $\CF$ be the set of  forests composed of trees in $ \CT $.  The set $ \CF_N $ corresponds of forests of size $ N $ in the sense that $ \tau_1 \cdots \tau_n \in \CF_N$ satisfied $ \sum_{i=1}^n |\tau_i| \leq N $.  Any rooted tree $\tau \in \CT$, different from the empty tree, $\one$, can be written in terms of the $B^{i}_+$-operators, $ i \in  \{0,...,d\}$. Indeed,  we have that $\tau = B^{i}_+(\tau_1 \cdots \tau_n)$, which connects the roots of the trees in the forest $\tau_1 \cdots \tau_n \in \CF$ to a new root decorated by $ i $.
We define $\CH = \langle \CF \rangle$ as the linear span of $\CF$. One can endow this vector space with a Hopf algebra structure where the product is given by the forest product. The coproduct is given by the Butcher-Connes-Kreimer coproduct:
 \begin{equs} \label{Connes-Kreimer}
 \Delta (\tau) 
	= \mathbf{1} \otimes \tau + (B^{i}_+ \otimes \mathrm{id})\Delta (\tau_1 \cdots \tau_n).
 \end{equs}
We denote by $ \CG $ (resp. $ \CG_N $) the set of characters from the Hopf algebra $\CH $ (resp. $\CH_N = \langle \CF_N \rangle $) into $\R$. These are linear algebra morphisms forming a group with respect to the convolution product $ \star $ with inverse given by the antipode $ \CA $
\begin{equs}
\label{convolution}
	X \star Y :=  (X \otimes Y ) \Delta, \qquad X^{-1} = X \circ \CA 
\end{equs}
The unit for the convolution product is the co-unit $\mathbf{1}^*$ which is non zero only on the empty tree.
Let $\gamma\in\,]0,1[$, a branched $\gamma$-rough path is a path $X:[0,1]^2\to\cal G$ such that $X_{tt}=\one^{*}$, it satisfies Chen's rule
  \begin{equs} \label{chen}
  X_{su} \star X_{ut}=X_{st}, \qquad s,u,t\in[0,1],
  \end{equs}
  and the analytical condition
  \begin{equs} \label{analytical}
  |\langle X_{st},\tau\rangle|\lesssim|t-s|^{\gamma|\tau|},
  \end{equs}
  for every $ \tau $ which does not contain the decorations zero on the nodes. Otherwise, we have
  \begin{equs}
  \sup_{0 \leq s,t \leq 1} \frac{\langle  X_{st}, \tau \rangle}{|t-s|^{(1-\gamma)|\tau|_0 + \gamma |\tau|}} < \infty,
  \end{equs}
  where $ |\tau|_0 $ counts the number of times the decoration $ 0 $ appears in $ \tau $.
  This extra assumption is needed when one wants to consider the renormalisation in \cite{BCFP}. Nodes with $ 0 $ decorations are distinguished and it is were some renormalisation may have occurred. In the sequel, we will consider the biggest $ N \in \N $ such that  $ \gamma N \leq 1 $. The branched $ \gamma $-rough paths are taking values in $ \CG_{N} $.
 We denote this space by $\BR^\gamma$.
  
  We are supposed given an alphabet $ A $ and we consider the linear span of the words on this alphabet denoted by $ T(A) $. We set $ \varepsilon  $ as the empty word. The product on $ T(A)$ is the shuffle product defined by
  \begin{equs}
  \varepsilon \shuffle v=v\shuffle \varepsilon  =v, \quad (au\shuffle bv) = a(u\shuffle bv) + b(au\shuffle v)
  \end{equs}
for all $u,v\in T(A)$ and $a,b\in A$.
The coproduct $\bar\Delta:T(A)\to T(A)\otimes T(A)$ is the deconcatenation of words: 
\[ \bar\Delta(a_1\dotsm a_n) = a_1\dotsm a_n\otimes \varepsilon  + \varepsilon \otimes a_1\dotsm a_n + \sum_{k=1}^{n-1}a_1\dotsm a_k\otimes a_{k+1}\dotsm a_n.\]
Equipped with this product and coproduct $ T(A) $  is a Hopf algebra. The grading of $ T(A) $
 is given by the length of words $\ell(a_1\dotsm a_n) = n$. We denote by $ \CG_A $ the group of characters associated to $ T(A) $ and by $ * $ the convolution product.
  An anisotropic $\gamma$-rough path, with $\gamma=(\gamma_a, \, a\in A)$, $ 0<\gamma_a<1$, is a map $X:[0,1]^2\to \CG_{A}$ such that $ X_{tt} = \varepsilon^{*} $ where $  \varepsilon^{*}$ is the counit. It satisfies
\begin{equs}
X_{su} * X_{ut}=X_{st}, \qquad |\langle X_{st},v\rangle|\lesssim|t-s|^{\hat{\gamma}\omega(v)}
\end{equs} 
for all $(s,u,t)\in[0,1]^3$ and word $v$. Moreover, one has $\hat{\gamma}=\min_{a\in A}\gamma_a$. For a word $v=a_1\dotsm a_k$ of length $k$ we define
\begin{equation}
  \label{eq:weight}
  \omega(v)=\frac{\gamma_{a_1}+\dotsc+\gamma_{a_k}}{\hat{\gamma}}=\frac{1}{\hat{\gamma}}\sum_{a\in A} n_a(v)\gamma_a
\end{equation}
where $n_a(v)$ is the number of times the letter $ a $ appears in $ v $.
The different weights $ \gamma_a $ correspond to a rough SDEs whose drivers have various regularities. One wants to incorporate them in this analytical bound. 
 We denote by $ \AN^{\gamma} $ the space of anisotropic $\gamma$-rough paths introduced in \cite{TZ}. A similar concept has been considered in \cite{Gyu16} called $ \Pi $-rough path. The idea of such paths has its roots in the foundation paper \cite{Lyons98}. 
Classical geometric rough paths are when the $ \gamma_a $ are all equal to the same $ \gamma $. 
 
 In the sequel, 
 the alphabet $ A $ will be either $ \CT $ or a subset of $\CT$. Then, the weight $ \gamma_{\tau} $ will correspond to the analytical bounds of a branched $\gamma$-rough path: $ \gamma |\tau| $ or $ (1-\gamma)|\tau|_0 +  \gamma |\tau|  $ depending on whether $ \tau $ contains $ 0 $ decorations.
 Then as for the branched rough paths, we will perform a truncation and consider paths taking values in $ \CG_{\CT_N,N}$. Elements of $ \CG_{\CT_N,N}$  are characters over $ T_N(\CT_N) $ which are  words $ v = \tau_1 \otimes \ldots \otimes \tau_n $ built on the alphabet $ \CT_N $ such that $  \sum_{i=1}^n |\tau_i| \leq N $. 
 
 The next theorem first stated in \cite{LV07} and reformulated in \cite{TZ} constructs an anisotropic rough path over a path $ (x^{a})_{a \in A} $:

  \begin{theorem}[Lyons-Victoir extension]
Let $(x^a)_{a \in A}$, with $x^a\in C^{\gamma_a}([0,1])$ such that $ 1 \notin \sum_{a \in A} \gamma_a \N $, then there exists an anisotropic rough path  $\bar X$ over $(x^a)_{a \in A}$: $ x_{t}^{a} - x_{s}^{a} =  \langle \bar X_{st},a \rangle $. 
  \label{thm:Lyons-Victoir}
\end{theorem}

The rough path constructed from the Lyons-Victoir theorem is neither unique nor canonical. At each step of the construction, arbitrary choices are made. An analogue construction exists in regularity structures for the reconstruction theorem where uniqueness is lost for negative exponent (see \cite{reg}). The Lyons-Victoir extension is at the core of the transformation from branched rough paths to anisotropic rough paths described in \cite{TZ}. Before, we need a map to transform trees into words which is the Hairer-Kelly map $ \psi $.
It has been introduced in \cite{HK15} and a reformulation of this map  is given in \cite[Def. 4 Sec. 6]{Br173} by 
\begin{definition}[Hairer--Kelly map]\label{HKmap}
The map $\psi :\mathcal{H} \rightarrow (T(\mathcal{T}), \shuffle)$ is defined as the unique Hopf algebra morphism from $\mathcal{H}$ to the shuffle Hopf algebra $(T(\mathcal{T}), \shuffle)$ obeying
\[
	\psi = (\psi\otimes P_{\one}) \circ \Delta,
\]
where  $P_{\one}:=\mathrm{id}-\one^{*}$ is the augmentation projector.
\end{definition}

\begin{remark} This definition using the Connes-Kreimer coproduct is very useful for performing proofs and reveals also the intrinsic construction of this map via a  recursion.
\end{remark}

The following theorem given in \cite{TZ} is an extension to anisotropic rough paths of the original theorem in \cite{HK15} on geometric rough paths.

\begin{theorem}\label{Hairer-Kelly}
  Let $X$ be a branched $\gamma$-rough path. There exists an anisotropic geometric rough path $\bar{X}$ indexed by words on the alphabet $\CT_N$, $ N = \lfloor 1/\gamma \rfloor $, with exponents $(\gamma_\tau, \tau\in\CT_N)$, and such that $\langle X,\tau\rangle = \langle\bar{X},\psi(\tau)\rangle$.
\end{theorem}
Using the previous theorem, the authors in \cite{TZ} define a new renormalisation by first noticing that the value of $\langle X,\tau\rangle$ can be modified by adding
the increment of a suitable $ \CC^{\alpha} $ function for a well-chosen $ \alpha $. They consider the following abelian group (under pointwise addition)
\[
  \cal C^\gamma:=\{(g^\tau)_{\tau\in \CT_N}: \, g^\tau_0=0, \, g^\tau\in C^{\gamma_\tau}([0,1]), \, \forall\, \tau\in\CT, |\tau|\leq N\}.
\]
Equipped with this family of maps, they are able to state one of their main results:

\begin{theorem}  \label{thm:Lyons-Victoirb}
  There exists a transitive free action of $\CC^\gamma$ on branched $ \gamma $-rough paths, a map $(g,X)\mapsto gX$ such that
  \begin{enumerate}
\item  For each $g,g'\in\CC^\gamma$ and $X\in\BR^\gamma$ the identity $g'(gX) = (g+g')X$ holds.
\item For every pair $X,X'\in\BR^\gamma$ there exists a unique $g\in\CC^\gamma$ such that $gX=X'$.
\end{enumerate}
\end{theorem}

This action is constructed using the Hairer-Kelly map $\psi$ via the Theorem~\ref{Hairer-Kelly} which relies on the Lyons-Victoir extension theorem. Indeed, $ gX $ is defined for every tree $ \tau $ by 
\begin{equs}
\langle gX_{st},\tau\rangle :=\langle g\bar{X}_{st},\psi(\tau)\rangle,
\end{equs}
where	$ g\bar{X} $ is the anisotropic geometric rough path given by Theorem~\ref{thm:Lyons-Victoir} over the path $ g^{\tau} + x^{\tau} $,  $ x^{\tau}_t -  x^{\tau}_s =  \langle  \bar X_{st},\tau\rangle  $.  The initialisation of the action is performed by sending $ X $ to $  \bar X$. Then, by acting with an element in $ \CC^{\gamma} $, we reconstruct the entire path via the Lyons-Victoir extension theorem. 
So each time, we act with a map $ g $, we reconstruct the path. We recall how the additive structure can be obtained.
Let $ g, \bar g \in \CC^{\gamma} $,  $ \bar g(gX) $ is the anisotropic rough path over $ \bar g^{\tau} + ( g x)^{\tau} $ where  $ ( g x)^{\tau}_t -  ( g x)^{\tau}_s =  \langle   g X_{st},\tau\rangle  $. One has 
\begin{equs}
 \langle \bar g(g X)_{st}, \tau \rangle  =  \bar g^{\tau}_{t} - \bar g^{\tau}_{s} +  \langle  (g X)_{st},\tau\rangle 
\end{equs}
 We apply the definition again to $ gX $ and we get that
\begin{equs}
 \langle g X_{st}, \tau \rangle  =  g^{\tau}_t - g^{\tau}_s + \langle  \bar  X_{st},\tau\rangle  
 \end{equs}
  which gives the additive structure of Theorem~\ref{thm:Lyons-Victoirb}.  
This construction seems to capture any renormalisation on the space of branched rough paths. Therefore, it is a natural question to see what are the $ g $ for various known renormalisations.

An alternative approach given in \cite{BC2019}  constructs a bijection between the two spaces 
$ \BR^{\gamma} $ and $ \AN^{\gamma} $. The main idea is to use an algebraic result from \cite{Foi02,cha10}: 
There exists a subspace $ \CB = \langle\tau_1, \tau_2,...\rangle $ of $ \CT $ such that $ \CH $ is isomorphic as a Hopf algebra to the tensor Hopf algebra $ T(\CB) $. Therefore, $ \CH_{N} $
is isormophic to some $ T_N(\CB_{N}) $, $ \CB_N  $ being a subspace of $ \CT_N $. 
This means that every $ \tau \in \CH_N $ has a unique representation of the form:
\begin{equs} \label{uniquedec}
\tau = \sum_{R} \lambda_R \tau_{r_1} \star \ldots \star \tau_{r_n}
\end{equs}
where the sum is performed over all the multi-indexes $ R = (r_1,\ldots,r_n) $ 
for which $ \sum_{i} |\tau_{r_i}| \leq N $.
Then, one can exhibit an isomorphism $ \Psi $ between the two spaces  $ \CH_{N} $ and $ T_N(\CB_{N}) $ based on the basis $ \CB_N $ (see \cite[Lemma 4.2]{BC2019}):
\begin{equs}
\Psi : \tau_1 \star \ldots \star \tau_r \mapsto
\tau_1 \otimes \ldots \otimes \tau_r
\end{equs}
where $ \tau_1 \otimes \ldots \otimes \tau_n \in T_N(\CB_{N}) $ . Then, the authors exploit this isomorphism to give their main result which is an isomorphism between non-geometric and geometric rough paths
\begin{theorem} \label{chiso} Let $ X \in \BR^{\gamma} $, then $ \tilde{X} := \Psi(X) \in \AN^{\gamma}$.
\end{theorem}
This result gives a canonical way to move from one representation to the other and does not depend on the choice of the basis $ \CB_N $ if one assumes that $ |\tau_i| \leq |\tau_j| $ for $ i \leq j $. It also avoids the use of the Lyons-Victoir theorem which reconstructs entirely the path on a different states space in a non-canonical way. One will also see in the sequel that the renormalisation behaves nicely toward this isomorphism whereas it is unclear how it can commute with the Hairer-Kelly approach.

\section{Interaction of the renormalisation with the Hairer-Kelly map }
\label{sec::3}

In this section, we consider two renormalisations on branched rough paths and see how they commute with the Hairer-Kelly map. They are both part of the  same family of maps that we will first introduce.
 We want to act on a branched rough paths with linear maps $ M : \CH \rightarrow \CH $ mutiplicative for the forest product. Given $ X \in \BR^{\gamma} $ and $ \tau \in \CT $, we set
\begin{equs}
\langle \hat X_{st},\tau\rangle := \langle  X_{st}, M \tau\rangle  =  \langle M^{*} X_{st}, \tau\rangle 
\end{equs}
where $ M^{*} $ is the adjoint of the map $ M $. Now, it remains to assume sufficiently nice properties on $ M $ in order to get the Chen's relation \eqref{chen} and the good analytical bounds \eqref{analytical}. One has by definition
\begin{equs}
 (\hat X_{su}\star \hat X_{ut}) = \left( X_{su} \circ M  \otimes   X_{ut} \circ M \right) \Delta
\end{equs}
If one assumes the cointeraction property 
\begin{equs} \label{cointeraction}
\left( M \otimes M \right) \Delta  = \Delta M
\end{equs}
Then
\begin{equs}
(\hat X_{su} \star \hat X_{ut}) & \left( X_{su} \circ M  \otimes   X_{ut} \circ M  \right) \Delta =  \left( X_{su}   \otimes   X_{ut}   \right) \Delta M  
\\ & = ( X_{su} \star  X_{ut})\circ M  =  X_{st} \circ M  =\hat X_{st}
\end{equs}
where we have applied the Chen's identity on $ X $. For the analytical bounds \eqref{analytical}, we have to assume that $ M $ sends a tree $ \tau $ to more regular terms. 
We therefore suppose that for every $ \tau \in \CT$, one has:
\begin{equs} \label{analytical conditions}
M \tau = \sum_i \lambda_i \tau_i, \quad \lambda_i \in \R, \, \tau_i \in \CT, \, \gamma_{\tau_i} \geq  \gamma_{\tau}, \, |\tau_i| \leq |\tau|.
\end{equs}
The last condition $ |\tau_i| \leq |\tau| $ guarantees that $ M $ respects the projection onto $ \CH_N $.

Given a linear map $ M $ satisfying the properties stressed before, we want to find a linear map $ \bar M : T(\CT) \rightarrow T(\CT) $ such that the following diagram commutes:

\begin{equ} \label{cdiag}
\begin{tikzcd}
\mathcal{H}  \arrow[r, "\psi"] \arrow[d, "M"]
& T(\mathcal{T}) \arrow[d, "\bar M"] \\
\mathcal{H}  \arrow[r,  "\psi"]
& T(\mathcal{T})
\end{tikzcd}
\end{equ}
We want also this map to satisfy: $ M \one = \one $ which together with the analytical bounds imply that $ M \one^{*} = \one^{*} M $.
 Moreover, we want this map to act on
the space of anisotropic rough paths $\AN^{\gamma}$ defined on $ T(\CT) $.
\begin{proposition} \label{barM} The map $ \bar M $ which makes the diagram~\ref{cdiag} commute is given for $ u, v \in T(\CT) $ and $ \tau \in \CT $ by:
\begin{equs} \label{defbarM}
\bar M (u  v) = (\bar M u)(  \bar M v),\quad \bar M  \tau = M \tau. 
\end{equs}
Moreover, $ \bar M $ acts on $ \AN^{\gamma} $, for every $ X \in  \AN^{\gamma}  $, $ X \circ \bar M \in  \AN^{\gamma}  $.
\end{proposition}
\begin{proof} We proceed by induction on the size of the trees. Let $ \tau \in \mathcal{T} $, then  one has from the cointeraction property \eqref{cointeraction}
\begin{align*}
\psi(M \tau) & = (\psi\otimes P_{\one})  \Delta M \tau  =  (\psi M \otimes P_{\one} M )  \Delta \tau
\end{align*}
Then we apply the induction hypothesis to get $ \psi M = \bar M \psi $. Indeed, the projection $ P_{\one} $ guarantees that $ \psi M $ is applied on terms smaller than $ \tau $. Then we use the fact that   $ M $ and $ \bar M $ coincide on trees to get
\begin{align*}
\psi(M \tau)
 = (\bar M \psi \otimes \bar M P_{\one})  \Delta \tau  = \bar M (\psi \otimes P_{\one})  \Delta \tau = \bar M \psi(\tau)
\end{align*}
The commutation between $ M $ and $P_{\one} $ is guaranteed by the fact that $ M $ commutes with the counit $ \one^{*} $. Let $ X \in \AN^{\gamma} $, the character property of $ X \circ \bar M$ follows from the fact that $ \bar M  $ respects the concatenation product and therefore the shuffle product. For the Chen's relation, we need to check a similar cointeraction property as in \eqref{cointeraction} where the coproduct $ \Delta $ is replaced by $ \bar \Delta $:
\begin{equs}
\left( \bar M \otimes \bar M \right) \bar \Delta  = \bar \Delta \bar M
\end{equs}
Such identity is straightforward to check because the deconcatenation coproduct does not act on the letters. The analytical bounds are a consequence of the condition \eqref{analytical conditions} put on $ \bar M $ for $ \BR^{\gamma} $. Indeed, given a word $ v =  \tau_1 \otimes \ldots \otimes \tau_n $, one has
\begin{equs}
\bar M v = \sum_{i_1, \ldots, i_n} \lambda_{i_1}\ldots \lambda_{i_n} \tau_{1,i_1} \otimes \ldots
\otimes \tau_{n,i_n}, \quad \tau_{k,i_k} \in \CT, \, \gamma_{\tau_{k,i_k}} \geq \gamma_{\tau_k}
\end{equs}
where $ \bar M \tau_{k}  = \sum_{i_k} \lambda_{i_k} \tau_{k, i_k}  $.
Therefore, 
\begin{equs}
 w(  \tau_{1,i_1} \otimes \ldots
\otimes \tau_{n,i_n})  \geq  w(v)
\end{equs}
which allows us to conclude.
\end{proof}

\begin{remark}
By considering the framework of Regularity Structures, one can define a deformed version of the Hairer-Kelly map by replacing $ \Delta $ by the coproduct given in \cite{BHZ} for the positive renormalisation. This map will give all the terms produced by the twisted antipode and they will be ordered through the tensor product. Indeed, the root is located at the rightmost letter and the partial order on the edges cut in the tree is preserved by the shuffle product.
\end{remark}

We want to describe the adjoint of $\bar  M $ as a translation map following the formalism in \cite[Section 2]{BCFP}.

\begin{proposition} The adjoint $ \bar M^{*} $ of $ \bar M $ is given as a translation map:
\begin{align*}
  \bar M^{*}  \tau =  \sum_{\tau_1} C(\tau,\tau_1) \tau_1
\end{align*}
where the sum is performed over $ \tau_{1} $
such that $C(\tau,\tau_1) :=    \langle M \tau_1, \tau \rangle$.
\end{proposition}

\begin{remark} In the case of the translation of rough paths, the transformations which have been considered are the ones which translate only one letter. Here, this is an example where the translation occurs on many letters.
\end{remark}

\begin{remark}
One can try to replace the Hairer-Kelly map $ \psi $ by the arborification map $ \mathfrak{a} $ which is a natural algebra morphism between the forests and the words. Then one cannot find interesting maps $ \bar M $ such that the diagram \eqref{cdiag}  commutes. Indeed, as notice in \cite{Br173} the map $ \mathfrak{a} $ is described by
\begin{align*}
\mathfrak{a} = (\mathfrak{a} \otimes P_{\bullet}) \Delta,
\end{align*}
where $ P_{\bullet} $ is the projector on the tree composed of only one node. If we try to repeat the steps of the previous proof, we get:
\begin{align*}
\mathfrak{a}(M) & = (\mathfrak{a} \otimes P_{\bullet}) \Delta M = (\mathfrak{a} M  \otimes P_{\bullet}  M   )  \Delta
\end{align*}
Now we cannot identify a non trivial map $ \bar M   $ such that $ P_{\bullet}  M  = \bar M  P_{\bullet}  $. Indeed for $ \tau $ having more than one node,
\begin{align*}
 P_{\bullet}  M \tau = \sum_{i} C(\tau, \bullet_i) \bullet_i, \quad  P_{\bullet} \tau = 0.
\end{align*}
Therefore, one needs $ C(\tau, \bullet_i) $ to be equal to zero in order to guarantee such commutation. This is rather a strong constraint and excludes the renormalisation considered in Section~\ref{sec::3.1}
\end{remark}

Another diagram  of interest is the one obtained by replacing the Hairer-Kelly map by the isomorphism given by Foissy and Chapoton: 
\begin{equ} \label{cdiag2}
\begin{tikzcd}
\mathcal{H}^{*}_N  \arrow[r, "\Psi"] \arrow[d, "M^{*}"]
& T_N(\mathcal{B}_N) \arrow[d, "\tilde M^{*}"] \\
\mathcal{H}^{*}_N  \arrow[r,  "\Psi"]
& T_N(\mathcal{B}_N)
\end{tikzcd}
\end{equ}
This time we have the dual point of view and consider the space $ \CH_N^{*} $ and $ T_{N}(\CB_{N}) $. 
One can try to find a map $ \tilde M^{*} $ which makes this diagram commute. We consider  
the space of anisotropic rough paths $\AN^{\gamma}$ defined now on $ T_N(\CB_N) $. We proceed as the same as before and get:
\begin{proposition} \label{tildeM} The map $ \bar M $ which makes the diagram~\ref{cdiag2} commute is given for $ uv \in T_N(\CB_N) $ and $ \tau \in \CB_N $ by:
\begin{equs} \label{deftildeM}
\tilde M^{*} (u  v) = (\tilde M^{*} u)(  \tilde M^{*} v),\quad \tilde M^{*}  \tau = \Psi(M^{*} \tau). 
\end{equs}
Moreover, $ \tilde M $ acts on $ \AN^{\gamma} $, for every $ X \in  \AN^{\gamma}  $, $ \tilde M^{*} X  \in  \AN^{\gamma}  $.
\end{proposition}
\begin{proof} For proving \eqref{tildeM}, we proceed by induction.  Let $ \tau \in \mathcal{H}_N $, we consider it as a linear functional on  $ \CH_N^{*} $ such that : $ <\tau,\sigma> = 1 $ if $ \sigma = \tau $ and zero elsewhere. Then one has from \eqref{uniquedec}:
\begin{equs}
\tau = \sum_{R} \lambda_R \tau_{r_1} \star \ldots \star \tau_{r_n}.
\end{equs}
 By the the cointeraction property \eqref{cointeraction}
\begin{equs}
M^{*}  \tau = \sum_{R} \lambda_R  (M^* \tau_{r_1}) \star \ldots \star (M^* \tau_{r_n})
\end{equs}
Then every $ M^{*} \tau_{r_i} $ can be expressed using the basis $ \CB_N $.
Therefore, one can conclude that
\begin{equs}
\Psi(M^{*} \tau) & = \sum_{R} \lambda_R  \Psi(M^{*} \tau_{r_1}) \otimes \ldots \otimes \Psi(M^{*} \tau_{r_n}) \\
& = \tilde M \left(    \sum_{R} \lambda_R  \tau_{r_1} \otimes \ldots \otimes \tau_{r_n} \right)
\end{equs}
As before, it can be viewed as a translation map. The translation occurs on the $ \tau_{r_i} $ which will be rewritten by applying $ M^{*} $ on them. The analytical bounds follow again from the condition \eqref{analytical conditions}.
\end{proof}

\begin{remark}The map $ \tilde M $ needs to express the renormalisation in the basis given by the isomorphism $ \Psi$ which is less straightforward than the Hairer-Kelly map. The cost to pay in the Hairer-Kelly approach is the use of an extended alphabet by considering all the trees in $ \CT_N $.
The advantage of the isomorphism approach is revealed in the last section of the paper where we are able to observe a nice commutation with the renormalisation and the construction of an anisotropic rough path from a branched rough path. Such result for the Lyons-Victoir extension theorem remains unclear. Before, we review the main remornalisations for rough paths that satisfy the cointeraction property.
\end{remark}

\subsection{BPHZ renormalisation}

\label{sec::3.1}

The BPHZ renormalisation has been introduced for renormalising Feynman diagrams. It appear naturally in  the context of branched rough paths as stressed in \cite{BCFP}. Some examples are given in \cite{BCF} where the need for renormalisation is highlighted in the context of singular SDEs. The idea is to construct $ M $ via an extraction/contraction map named $ \Deltam $. This map is given by
\begin{equs} \label{coproduct_extraction}
\Deltam \tau = \sum_{\one \subset \tau_1 \ldots \tau_n \subset \tau}  \tau_1 \ldots \tau_n \otimes \tau / \tau_1 \ldots \tau_n 
\end{equs}
where the sum is performed over all sub-forests of $ \tau $ which are disjoint sub-trees of $ \tau $. Then, the sub-forest is located on the right hand side of the tensor product. On the left, we contract the sub-forest inside $ \tau $ to a single node decorated with $ 0 $. 
 We consider a character $ v : \CH \rightarrow \R $ multiplicative  for the forest product, being zero on trees containing $0$ decorations. The renormalisation map $ M_v $ is defined in \cite{BCFP} by:
\begin{equs}
M_v = \left(v  \otimes \id  \right) \Deltam
\end{equs}
It turns out that $ \Deltam  $ cointeracts with the Butcher-Connes-Kreimer coproduct $ \Delta $ satisfying the following identity:
\begin{equs} \label{cointeract}
\mathcal{M}^{(13)(2)(4)}\left(\Deltam \otimes \Deltam \right) \Delta = \left( \id \otimes \Delta \right) \Deltam
\end{equs}
where 
\begin{equs}
\mathcal{M}^{(13)(2)(4)}(\tau_1 \otimes \tau_2 \otimes \tau_3 \otimes \tau_4) = \tau_1 \cdot \tau_3 \otimes \tau_2 \otimes \tau_4.
\end{equs}
This cointeraction has been observed on similar structures without any decorations in \cite{CA}. It is also the crucial property needed at the level of regularity structures (see \cite{BHZ}).
The identity \eqref{cointeract} and the character property of $ v $ are enough for checking \eqref{cointeraction}. Indeed, one has
\begin{equs}
\left( M_v \otimes M_v \right) \Delta & =   \left(  \left(v  \otimes \id  \right) \Deltam \otimes \left(v  \otimes \id  \right) \Deltam \right) \Delta 
\\ & = \left( v \otimes \id \otimes \id \right) \mathcal{M}^{(13)(2)(4)}\left(\Deltam \otimes \Deltam \right) \Delta
\\ & = \left( v \otimes \id \otimes \id \right)
\left( \id \otimes \Delta \right) \Deltam
\\  & =   \Delta M_v
\end{equs}

For the analytical bounds, one can observe that  the term  $ \bar \tau =  \tau/\tau_1 \ldots \tau_n $ is such that $ \gamma_{\bar \tau} \geq \gamma_{\tau} $ when the $ \tau_i $ do not contain any $ 0 $ decorations.

\subsection{Local products renormalisation}

A more general renormalisation has been introduced in \cite{Br17} and  it was also used in \cite{CMW19} for a priori bounds in the entire subcritical regime of the model $ \phi^{4}_{4-\kappa} $. The idea is to construct the renormalisation  as a recursive formula when one iterates a map $ R $ having certain good properties. This map implements  renormalisation on ill-defined distributional products and then it is iterated deeper in the tree.  
One can derive this map in the simple case of branched rough paths. This derivation is new and offers a new family of renormalisation maps in this context. We consider linear maps $ R : \CT \rightarrow \CT $ satisfying:

\begin{enumerate}
\item For each $ \tau \in \mathcal{T} $ there exist $ \tau_i \in \CT$ such that, 
\begin{equs}
R \tau = \tau + \sum_i \lambda_i \tau_i, \quad \gamma_{\tau_i} \geq \gamma_{\tau}, \quad |\tau_i| < |\tau|
\end{equs}
 \item One has $ \left( R \otimes \id \right) \Delta = \Delta R $
\end{enumerate}

We denote by $ \mathcal{L}_{ad}(\CT) $ the set of admissible maps satisfying the previous properties. For $ R \in \mathcal{L}_{ad}(\CT) $, we define a renormalisation map $M=M_R$ by:
\begin{equ}[e:defM]
  \left\{ \begin{aligned}
  & M \one =  M^{\circ} \one  = \one,  \qquad  M \tau \bar{\tau} = 
 \left( M \tau \right) \left(
 M \bar{\tau} \right)  
    \\
  & M^{\circ} B^{i}_+(\tau) =  B^{i}_+(M \tau), \qquad M  B^{i}_+(\tau) = M^{\circ} R \, B^{i}_+(\tau),
  \end{aligned} \right.
  \end{equ}
where $ \tau, \bar \tau \in \CH $.  Such map is well-defined because at each step $ R $ sends a tree $ \tau $ to trees with less nodes. 
 The key property which remains to be proved is the cointeraction see Proposition~\ref{Rcointeraction} below. This result is new in itself and certainly not true when we consider this renormalisation at the level of SPDEs except if one can guarantee that  $ \gamma_{\tau} \leq 1 $ on planted tree.
This implies that we cannot get any derivatives when we apply a deformed version of the Butcher-Connes-Kreimer coproduct $ \Delta $.
This specific property has already been noticed in \cite[Remark 45]{BCFP}.
\begin{proposition} \label{Rcointeraction}
One has the following cointeractions:
\begin{equs}
\Delta M = \left( M \otimes M^{\circ} \right) \Delta, \quad \Delta M^{\circ} = \left( M^{\circ} \otimes M^{\circ} \right) \Delta
\end{equs}
\end{proposition}
\begin{proof} We proceed by induction on the size of the trees. Let $ \tau =  B^{i}_+(\bar \tau) \in \CT $, one has
\begin{equs}
\Delta M \tau & = \Delta M^{\circ} \left(R \tau- \tau \right) + \Delta M^{\circ} \tau.
\end{equs}
Then by applying the induction hypothesis on $ R \tau - \tau $, one gets:
\begin{equs}
\Delta M^{\circ} \left(R \tau- \tau \right) & = 
\left( M^{\circ} \otimes M^{\circ} \right) \Delta \left( R \tau - \tau \right)
\\ & = \left( M^{\circ} R \otimes M^{\circ} \right) \Delta  \tau -\left( M^{\circ} \otimes M^{\circ} \right) \Delta \tau 
\\ & = \left( M \otimes M^{\circ} \right) \Delta  \tau -\left( M^{\circ} \otimes M^{\circ} \right) \Delta \tau. 
\end{equs}
One the other hand:
\begin{equs}
\Delta M^{\circ} \tau & = \Delta B_{+}^{i}(M \bar \tau)
\\ & =  \mathbf{1} \otimes  B_{+}^{i}(M \bar \tau) + (B^{i}_+ \otimes \mathrm{id})\Delta M \bar \tau 
\\ & = \one \otimes M^{\circ} \tau + \left(B^{i}_+ M \otimes M^{\circ} \right)\Delta \bar \tau 
\\ & = \one \otimes M^{\circ} \tau + \left( M^{\circ} B^{i}_+\otimes M^{\circ} \right)\Delta \bar \tau
\\ & = \left( M^{\circ} \otimes M^{\circ} \right) \Delta \tau
\end{equs}
which concludes the proof.
\end{proof}
\begin{remark}
It has been shown in the context of Regularity Structures (see \cite[Section 4]{Br17}) that the BPHZ renormalisation can be viewed as a specific case of the local products renormalisation. The idea is to choose $ R $ such  that it performs the extraction at the root and the map $ M^{\circ} $ extracts the other trees of the chosen sub-forest. One has the same property in the context of branched rough paths. We just need to replace the condition $ |\tau_i| < |\tau| $ by  $|\tau_i| \leq |\tau| $. When $ |\tau| = |\tau_i| $, we assume that $ |\tau|_0 < |\tau_i|_0 $. This total order allows us to conduct the induction.
\end{remark}
\section{Explicit formulae for renormalised branched rough paths}
\label{sec::4}
We have seen in the previous section how various renormalisations behave toward the change of structure moving from Hopf algebra on trees to Hopf algebra on words.  We investigate the commutation property when we move from  branched to anisotropic rough paths. For the entire section, we consider the space $ \BR^{\gamma} $ of branched $ \gamma$-rough paths and a linear map $ M : \CH_N \rightarrow \CH_N$ satisfying the properties given in Section~\ref{sec::3}. 

We first start with the construction coming from \cite{TZ}.  By applying Theorem~\ref{thm:Lyons-Victoir} in \cite[Sec. 7]{TZ}, the authors got the existence of a unique $ g \in \CC^{\gamma} $ such that:
\begin{equs}
  \langle X_{st}, M \tau  \rangle = \langle gX_{st}, \tau \rangle = \langle g\bar X_{st}, \psi(\tau) \rangle.
\end{equs}
Then by using the fact that
\begin{equs}
\langle g\bar X_{st}, \psi(\tau) \rangle =  g_{t}^\tau-g_s^\tau + \langle \bar X_{st}, \psi(\tau) \rangle 
\end{equs}
one can provide a recursive formula for $ g $:
\begin{equs} \label{recursive_g}
  g_t^\tau-g_s^\tau= \langle X_{st},M \tau\rangle - \langle \bar X_{st},\tau\rangle - \langle g\bar{X}_{st},\psi_{|\tau|-1}(\tau)\rangle.
\end{equs}
where $ \psi(\tau) = \tau + \psi_{|\tau|-1}(\tau) $. Therefore, $ g \bar X_{st} $ is applied to terms of lower orders in the right hands side of \eqref{recursive_g}.
An explicit formula relating the two renormalisations is missing. A first guess will be
\[
 g_t^\tau = x_{t}^{M \tau - \tau}, \quad x^\tau_t-x^\tau_s = \langle \bar X_{st},\tau\rangle,
\]
where $ x_t^{\cdot} $ is extended linearly to a linear combination of trees. This formula is checked for $ g=0 $ or $ M = \id $.
 It turns out that this guess may not be true and one has to add correction terms.

\begin{proposition} \label{formula}
The map $  g \in \CC^{\gamma}$ in \eqref{recursive_g} is defined recursively by:
\begin{align*}
 g_t^\tau-g_s^\tau= x_{t}^{M \tau - \tau} -  x_{s}^{M \tau - \tau} + \langle \bar M^{*} \bar{X}_{st},\psi_{|\tau|-1}(\tau)\rangle - \langle g\bar{X}_{st},\psi_{|\tau|-1}(\tau)\rangle.
\end{align*}
\end{proposition}

\begin{proof}
One has  from Theorem~\ref{Hairer-Kelly} and Proposition~\ref{barM}
\begin{equs}
\langle X_{st},M \tau\rangle & = \langle \bar X_{st},\psi(M \tau)\rangle 
 = \langle \bar X_{st}, \bar M \psi( \tau)\rangle
   = \langle \bar M^{*} \bar X_{st},  \psi( \tau)\rangle
\\ & =  \langle \bar M^{*} \bar X_{st},  \tau\rangle + \langle \bar M^{*} \bar X_{st}, \psi_{|\tau|-1}( \tau)\rangle.
\end{equs}
Then by plugging this expression into \eqref{recursive_g}, one gets
\begin{align*}
  g_t^\tau-g_s^\tau 
   & = \langle \bar M^{*} \bar X_{st},  \tau\rangle  - \langle \bar X_{st},\tau\rangle + \langle \bar M^{*} \bar X_{st}, \psi_{|\tau|-1}( \tau)\rangle - \langle g\bar{X}_{st},\psi_{|\tau|-1}(\tau)\rangle \\
  & = \langle \bar X_{st},  M  \tau - \tau\rangle + \langle  \bar M^{*} \bar X_{st},\psi_{|\tau | - 1}( \tau)\rangle - \langle g\bar{X}_{st},\psi_{|\tau|-1}(\tau)\rangle
\end{align*}
where we have used the following identities 
\begin{equs}
\langle \bar X_{st},  M \tau \rangle = \langle \bar M^{*} \bar X_{st},  \tau \rangle, \quad \bar M \tau = M \tau.
\end{equs}
We conclude by the fact that
\begin{align*}
 x_{t}^{M \tau - \tau}  - x_{s}^{M \tau - \tau} = \langle \bar X_{st},M \tau - \tau \rangle. 
\end{align*}
\end{proof}
\begin{remark}
The proposition~\ref{formula} can be rephrased as:
\begin{align*}
 g_t^\tau-g_s^\tau= x_{t}^{M \tau - \tau} -  x_{s}^{M \tau - \tau} + \text{difference on lower degree terms}
\end{align*}
Indeed, the rough paths $ g \bar X $ and $ \bar M_v^{*} \bar X $ do not necessary coincide outside the Hairer-Kelly map. This reveals a difference between the two renormalisation approaches.
\end{remark}

In fact, one can be more precise and give a non-recursive formula. We first recall that by going to the adjoint and by applying Theorem~\ref{Hairer-Kelly}, one gets
\begin{equs}
\langle X_{st},M \tau\rangle = \langle M^* X_{st},  \tau\rangle = \langle \overline{M^* X_{st}},  \psi(\tau) \rangle. 
\end{equs}

\begin{proposition} \label{ext} If $ g X $ and $ M^{*} X $ coincide then $ g \bar X $ and $ \overline{M^*  X} $ also coincide.
\end{proposition}
\begin{proof} We proceed by induction and suppose that $ g \bar X $ and $ \overline{M^{*} X} $ have been constructed on $ T(\CT_k) $ and that they coincide on this space. They will be denoted by $ g \bar X^{(k)} $ and $ \overline{M^{*} X}^{(k)}$. They have both been constructed iteratively over the same paths 
$ ((gx)^\tau : \tau \in \CT_k) $ and the application of the Lyons-Victoir extension theorem. For $ \tau \in \CT_{k+1} $, there exists a path $ (gx)^{\tau} $ such that
\begin{equs}
(gx)_t^{\tau} -  (gx)_t^{\tau} & = 
\langle g X_{st}, \tau \rangle -\langle g \bar X^{(k)}_{st}, \psi_{k}(\tau) \rangle \\
& = \langle  M^{*} X_{st}, \tau \rangle -\langle \overline{M^*  X_{st}}^{(k)}, \psi_{k}(\tau) \rangle
\end{equs}
This path is obtained by looking at the increments of $ \langle  M^{*} X_{st}, \tau \rangle  $ (see the proof of \cite[Theorem 5.6]{TZ}). The equality comes from the induction hypothesis and the fact that $ g X = M^* X $. Then the Theorem~\ref{thm:Lyons-Victoir} extends $ g \bar X $ and therefore $\overline{M^* X} $ to the same rough path on $ T(\CT_{k+1}) $.
\end{proof}

\begin{theorem} \label{formula1}
The map $  g \in \CC^{\gamma}$ in \eqref{recursive_g} is given by the formula:
\begin{equs} \label{formula11}
  g_t^\tau-g_s^\tau= \langle \overline{M^{*}  X_{st}},\tau \rangle - \langle \bar X_{st},\tau\rangle.
\end{equs}
\end{theorem}
\begin{proof}
One has
\begin{align*}
  g_t^\tau-g_s^\tau & = \langle X_{st}, M \tau\rangle - \langle \bar X_{st},\tau\rangle - \langle g\bar{X}_{st},\psi_{|\tau|-1}(\tau)\rangle \\
 & =  \langle M^{*} X_{st}, \tau\rangle - \langle \bar X_{st},\tau\rangle - \langle \overline{M^{*}  X_{st}},\psi_{|\tau|-1}(\tau)\rangle \\
 & = \langle \overline{M^{*}  X_{st}},\tau \rangle - \langle \bar X_{st},\tau\rangle.
\end{align*}
where for the second line we have used Proposition~\ref{ext}.
\end{proof}

\begin{remark}
The formula~\eqref{formula11} gives a new perspective on the additive property observed for the action of the space $ \CC^{\gamma} $. Indeed, let $ g, \bar g, \tilde g \in \CC^{\gamma} $ such that $ \tilde g X = \bar g (gX) $. Then one has
\begin{equs}
\tilde g_t^\tau- \tilde g_s^\tau & = \langle \tilde g \bar  X_{st},\tau \rangle - \langle \bar X_{st},\tau\rangle \\
& = \langle \tilde g \bar  X_{st},\tau \rangle -  \langle g \bar  X_{st},\tau \rangle + \langle g \bar  X_{st},\tau \rangle- \langle \bar X_{st},\tau\rangle
\\ & = \bar g^{\tau}_t - \bar g^{\tau}_s +  g^{\tau}_t -  g^{\tau}_s,
\end{equs}
where for the first line we apply \eqref{formula11}. Then in the second line, we make appear a telescopic sum and we conclude by applying \eqref{formula11} twice in the third line.
\end{remark}

\begin{remark}
Now if $ \overline{M^*  X_{st}}$ and $ \bar M^* \bar X_{st} $ coincide then we get
\begin{equs}
 g_t^\tau-g_s^\tau= x_{t}^{M \tau - \tau} -  x_{s}^{M \tau - \tau}.
\end{equs}
We need to prove that the extension theorem  and the renormalisation commute. Such result is rather unclear. Indeed, one has
\begin{equs}
x_t^{M \tau} -  x_t^{M \tau}    = 
\langle  X_{st}, M \tau \rangle -\langle \bar M^*  \bar X^{(|\tau|-1)}_{st}, \psi_{|\tau|-1}(\tau) \rangle,
\end{equs}
where the extension is applied in the construction of $ \bar X $. Then, 
 for $ M \tau = \sum_i \lambda_i \tau_i $, one gets
\begin{equs} \label{taui}
x_t^{ \tau_i} -  x_t^{ \tau_i} = 
\langle  X_{st}, \tau_i \rangle -\langle   \bar X^{(|\tau_i|-1)}_{st}, \psi_{|\tau_i|-1}(\tau_i) \rangle
\end{equs}
and the extension theorem is applied to each of the $ \tau_i $.
On the other hand,
\begin{equs}
 \langle  \overline{M^{*}  X_{st}},\tau \rangle  =
\langle  X_{st}, M \tau \rangle -\langle  \overline{M^*   X}^{(|\tau|-1)}_{st}, \psi_{|\tau|-1}(\tau) \rangle
\end{equs}
where the extension is applied to $ M \tau $, linear combination of the $ \tau_i $, which marks a clear difference with \eqref{taui}.
\end{remark}
We conclude by presenting the nice interaction observed with the isomorphism $ \Psi $ which is in contrast with the use of the Hairer-Kelly map:
\begin{theorem} \label{commute_iso} One has the following identity:
\begin{equs}
\Psi(M^{\star} X) =  \tilde{M}^{*} \Psi( X).
\end{equs}
\end{theorem}
\begin{proof}
This is just an application of Proposition~\ref{tildeM}.
\end{proof}

\bibliographystyle{./Martin}
 \bibliography{refs_projet_recherche}

\end{document}